\documentclass[12pt]{amsart}

\usepackage[pagewise]{lineno}

\usepackage{amsmath,amssymb,amsthm,amsfonts,mathrsfs}
\usepackage{fullpage}
\usepackage{relsize}
\usepackage{colonequals}

\usepackage[colorlinks=true,
linkcolor=blue,
anchorcolor=blue,
citecolor=red
]{hyperref}

\allowdisplaybreaks 
\newtheorem{theorem}{Theorem}[section]
\newtheorem{lemma}[theorem]{Lemma}
\newtheorem{corollary}[theorem]{Corollary}
\newtheorem{proposition}[theorem]{Proposition}

\theoremstyle{definition}

\theoremstyle{remark}
\newtheorem{remark}[theorem]{Remark}
\newcommand{\N}{\mathbb{N}}
\newcommand{\Z}{\mathbb{Z}}

\newcommand{\R}{\mathbb{R}}
\newcommand{\C}{\mathbb{C}}
\newcommand{\sS}{\mathscr{S}}
\newcommand{\sF}{\mathscr{F}}
\newcommand{\sH}{\mathscr{H}}
\newcommand{\cR}{\mathcal{R}}
\newcommand{\cA}{\mathcal{A}}

\newcommand{\ve}{\varepsilon}
\newcommand{\on}{\operatorname}
\renewcommand{\mod}[1]{\,(\on{mod}#1)}
\newcommand{\of}[1]{\left(#1\right)}
\newcommand{\set}[1]{\left\{#1\right\}}
\newcommand{\BEu}[1]{\underset{#1}{\mathlarger{\mathlarger{\mathbb{E}}}^{~}}\,}

\author[X. Su]{Xiang Su}
\address{School of Mathematics and Statistics, Yunnan University, Kunming, Yunnan 650500, China}
\email{suxiang8909@163.com}

\author[B. Wang]{Biao Wang}
\address{School of Mathematics and Statistics, Yunnan University, Kunming, Yunnan 650500, China}
\email{bwang@ynu.edu.cn}

\author{Shaoyun Yi}
\address{School of Mathematical Sciences, Xiamen University, Xiamen, Fujian 361005, China}
\email{yishaoyun926@xmu.edu.cn}

\thanks{Corresponding author: Biao Wang (bwang@ynu.edu.cn).}
\date{\today}

\makeatletter
\@namedef{subjclassname@2020}{\textup{}2020 Mathematics Subject Classification}
\makeatother

\title{Asymptotic uncorrelations between functions with squarefull kernel and functions of invariant average}

\begin{document}
	
\begin{abstract}
In 1986, Ivi\'c and Tenenbaum introduced arithmetic functions with squarefull kernel, which are also called $s$-functions. Later, Erd\H{o}s and Ivi\'c gave an asymptotic estimate on the shifted convolution sums of $s$-functions.  Recently, Bergelson and Richter studied the orbits along the prime Omega function in a uniquely ergodic topological dynamical system and established a new dynamical generalization of the prime number theorem (PNT). These orbits can be viewed as functions of invariant average under multiplications. In this paper, we show that both $s$-functions and their shifted convolutions are asymptotically uncorrelated to the orbits along the prime Omega function in a uniquely ergodic system. As a consequence, we obtain a refinement of the PNT via the local distribution of $s$-functions. Furthermore, several variants of these results are established as well.
\end{abstract}
\maketitle

\numberwithin{equation}{section}

\section{Introduction and statement of results}

In the analytic number theory, a fundamental theme is the study of  the correlation between two arithmetic functions. For any two arithmetic functions $a, b\colon \N \to \C$, they are called \textit{asymptotically uncorrelated} if 
\begin{equation*}
    \lim_{N \to \infty} \Big| \frac1N\sum_{n=1}^N a(n) \overline{b(n)} - \bigg(\frac1N\sum_{n=1}^N a(n) \bigg)\bigg( \frac1N\sum_{n=1}^N \overline{b(n)} \bigg) \Big|=0.
\end{equation*}
In 2010, Sarnak \cite{Sarnak2010ias,Sarnak2010} conjectured that any deterministic sequence is asymptotically uncorrelated to the M\"obius function. This conjecture gives a generalization of the prime number theorem (PNT) in dynamical systems. Let $(X, \nu, T)$ be a uniquely ergodic topological dynamical system, and denote by $C(X)$  the set of continuous functions on $X$. Let $h \in C(X)$ and $x\in X$. In 2022, to extend Sarnak's conjecture, Bergelson and Richter \cite{BergelsonRichter2022} established a new  dynamical generalization of the PNT, showing that
\begin{equation}\label{BR2022thmA}
		\lim_{N\to\infty}\frac1N\sum_{n=1}^N h(T^{\Omega(n)}x)=\int_X h \,d\nu,
\end{equation}
where $\Omega(n)$ denotes the number of prime divisors of $n$ counted with multiplicity. 
Let $\mu(n)$ be the M\"obius function. They also showed that
\begin{equation}\label{BR_mu}
		\lim_{N\to\infty}\frac1N\sum_{n=1}^N \mu^2(n)h(T^{\Omega(n)}x)=\frac6{\pi^2}\int_X h \,d\nu,
\end{equation}
which is another dynamical generalization of the PNT. 

Notice that $\mu^2(n)$ is the characteristic function of squarefree numbers, and the natural density of squarefree numbers is equal to $6/\pi^2$. By \eqref{BR_mu}, $\mu^2(n)$ and the orbit $\{h(T^{\Omega(n)}x) \}_{n\in \N}$ for any $x\in X$ are asymptotically uncorrelated. In general, one may ask if there are any other arithmetic functions $a:\N\to\C$ such that $a(n)$ is asymptotically uncorrelated to the orbits in \eqref{BR2022thmA}, that is, if $a(n)$ satisfies that
\begin{equation}\label{main_problem}
\lim_{N\to\infty}\frac1N\sum_{n=1}^N a(n)h(T^{\Omega(n)}x)=\lim_{N\to\infty}\frac1N\sum_{n=1}^N a(n)\cdot\int_X h \,d\nu
\end{equation}
for any $h\in C(X)$ and $x\in X$, provided the mean value of $a(n)$ exists.

 In 2023, Loyd \cite{Loyd2023} proved that \eqref{main_problem} holds if one takes $a(n)$ to be the arithmetic function in the Erd\H{o}s-Kac theorem. Later, in \cite{Wang2022ffa}, the second author proved that  \eqref{main_problem} holds if one takes $a(n)=1_{P^+(n)\in S}$ for any subset $S$ of primes of natural density in all primes,  where $P^+(n)$ denotes the largest prime factor of $n$. Here, for a statement $P$, the indicator symbol $1_{P}$ is equal to $1$ if $P$ is true and zero otherwise. In \cite{WWYY2025, LWWY2025aa}, the authors unified this result with Loyd's result. Recently, in \cite{Wang2025jnt, Wang2026BAustMS, WangYi2026jnt, Wang2025CMB, DengWang2026}, the authors proved that \eqref{main_problem} also holds if one takes $a(n)$ to be $1_{\Omega(n)-\omega(n)=k_1}$, $\mu^2(n(n+1))$, $\mu^2(n^2+1)$, $\mu_{k_2}(n)$  and $\frac{\varphi(n)}{n}$ for any integers $k_1\ge0$ and $k_2\ge2$, respectively, where $\omega(n)$ denotes the number of distinct prime divisors of $n$, and $\mu_{k_2}(n)$ denotes the M\"obius function of order $k_2$ defined by Apostol in \cite{Apostol1970}. More interesting examples of $a(n)$ such that \eqref{main_problem} holds can be found in these references. 
 
To establish these results on \eqref{main_problem}, the following invariant property of the orbits along $\Omega(n)$ plays a vital role. By \eqref{BR2022thmA}, using $T^{\Omega(dn)}x= T^{\Omega(n)}(T^{\Omega(d)}x)$ for any $d, n\in N$ and taking $T^{\Omega(d)}x$ as an initial point in $X$,  we have
\begin{equation}\label{inv_property}
\lim_{N\to\infty}\frac1N\sum_{n=1}^N h(T^{\Omega(dn)}x)=\int_X h \,d\nu
\end{equation}
for any $d\in \N$. In \cite{Wang2025jnt, LWWY2025aa}, this property is defined  for general arithmetic functions as follows. Let $c\colon \N\to\C$ be a bounded arithmetic function. We say that $c(n)$ is of \textit{invariant average under multiplications}, or \textit{invariant average} for short, if the mean value of $c(n)$ exists and satisfies 
\begin{equation}\label{eqn_inv_ave2}
    \lim_{N\to\infty}\frac1N\sum_{n=1}^N c(dn)= \lim_{N\to\infty}\frac1N\sum_{n=1}^N c(n)
\end{equation}
for all $d\in \N$. By \eqref{inv_property}, the orbit $\{h(T^{\Omega(n)}x)\}_{n\in \N}$ for any $x\in X$ is of invariant average $\int_X h \,d\nu$ under multiplications. By \cite[Corollary 1.6]{Xiao2025}, so is the orbit $\{h(T^{\Omega(\lfloor \alpha n +\beta\rfloor)}x)\}_{n\in \N}$ along Beaty sequences $\lfloor \alpha n +\beta\rfloor$, where $\alpha>0$ and $\beta\in \R$. Here, $\lfloor t \rfloor$ denotes the largest integer such that $0\leq t-\lfloor t \rfloor <1$. In this paper, we will consider a class of functions that are asymptotically uncorrelated to functions of invariant average.

Let $\sS$ be the set of all squarefree numbers, and let $\sF$ be the set of squarefull numbers (an integer $n$ is squarefull if $p^2\mid n$ whenever $p\mid n$), $1\in \sS$ and $1\in\sF$. Every integer $n\geq 1$ can be written uniquely as $n=qf$ with $q=q(n)\in \sS, f=f(n)\in \sF$ and $(q, f)=1$. We call $q=q(n)$ the squarefree part of $n$ and $f=f(n)$ the squarefull part of $n$. And we always use $f(n)$ to denote the squarefull part of $n$ in this paper. A nonnegative and integer-valued arithmetic function $a\colon \N\to\Z_{\ge0}$ is called a \textit{function with squarefull kernel}, or simply an \textit{$s$-function}, if $a(n) = a(f(n))$ for all $n\ge1$ and $a(n)\ll n^\ve$ for any $\ve>0$. It was introduced by Ivi\'c and Tenenbaum in \cite{IvicTenenbaum1986}\footnote{In \cite{IvicTenenbaum1986},  $a\colon\N\to\Z_{\ge0}$ is called an \textit{$s$-function} if $a(n) = a(f(n))$ for all $n\ge1$. In this paper, the condition  that $a(n)\ll n^\ve$ for any $\ve>0$ is included for $s$-functions without influence on our main results. In \cite{LvWang2020}, an $s$-function is also called a squarefull kernel function.}. For example, the number of nonisomorphic abelian groups of order $n$, the number of nonisomorphic semisimple rings with $n$ elements, $\mu^2(n)$, $\Omega(n)-\omega(n)$, $\Omega(n)-\omega_1(n)$ and $\omega_{l}(n), l\ge2$,  are all $s$-functions. Here, $\omega_k(n)$ denotes the number of prime factors of $n$ with multiplicity $k$ for $k\ge1$, which is introduced by Elma and Liu in \cite{ElmaLiu2022}. In the following theorem, we will show that all $s$-functions are asymptotically uncorrelated to bounded functions of invariant average.

\begin{theorem}\label{mainthm_disjoint}
	Let $b\colon\N\to\C$ be an $s$-function, and let $c(n)$ be a bounded arithmetic function of invariant average $A$ under multiplications. Then we have
	 	\begin{equation}\label{mainthm_disjoint_eqn}
		\lim_{N\to\infty}\frac1N\sum_{1\leq n\leq N} b(n) c(n) = \alpha \cdot A,
	\end{equation}
	where 
	\begin{equation}\label{dfn_alpha}
\alpha=\lim_{N\to\infty}\frac1N\sum_{1\leq n\leq N} b(n) =\frac6{\pi^2}\sum_{\substack{f\in \sF }} \frac{b(f)}{f} \prod_{p\mid f} \big(1+\frac1p\big)^{-1}.
\end{equation}
\end{theorem}

As an application of Theorem~\ref{mainthm_disjoint}, we give a refinement of \eqref{BR2022thmA} via the local density of $s$-functions as follows. Let $a(n)$ be a nonnegative and integer valued arithmetic function. Given an integer $k\ge0$, the \textit{local density} $d_k$ of $a(n)$ is defined as
\begin{equation*}
	d_k = \lim_{N\to\infty} \frac1N\sum_{\substack{n\leq N\\a(n)=k}}1,
\end{equation*}
provided this limit exists. For example, in 1947, Kendall and Rankin
\cite{KendallRankin1947} showed that the local densities of the number of nonisomorphic abelian groups of order $n$ exist. In 1955, R\'enyi \cite{Renyi1955} showed that all local densities for  $\Omega(n)-\omega(n)$ exist, and their generating function is given by
\begin{equation}\label{eqn_Renyi_d_k_gfcn}
	\sum_{k=0}^\infty d_kz^k=\prod_{p} \of{1-\frac{1}{p}}\of{1+\frac{1}{p-z}}
\end{equation}
for $|z|<2$, where the symbol $p$ runs over all prime numbers. In 1986, in a simple argument, Ivi\'c and Tenenbaum \cite{IvicTenenbaum1986} showed that all local densities for $s$-functions exist. They proved that if $a(n)$ is an $s$-function, then
\begin{equation}
	\sum_{\substack{n\leq N\\a(n)=k}}1 = d_k \cdot N + O(N^{\frac12}\log^2N)
\end{equation}
holds uniformly for $N\ge1$ and $k\ge0$, where $d_k$ is given explicitly as
\begin{equation}\label{def_d_k}
	d_k=\frac6{\pi^2}\sum_{\substack{f\in \sF\\a(f)=k} }  \frac1f \prod_{p\mid f} \big(1+\frac1p\big)^{-1}.
\end{equation}

Taking $b(n)=1_{\Omega(n)-\omega(n)=k}$ and $c(n)=h(T^{\Omega(n)}x)$ in Theorem~\ref{mainthm_disjoint}, we obtain \cite[Corollary 1.5]{Wang2025jnt}. In general, taking $b(n)=1_{a(n)=k}$ and $c(n)=h(T^{\Omega(n)}x)$ in Theorem~\ref{mainthm_disjoint}, we obtain the following refinement of Bergelson-Richter's theorem immediately. 

\begin{corollary}\label{thm_BR_KR}
Let $a(n)$ be an $s$-function. Let $(X, \nu, T)$ be a uniquely ergodic system. Then  we have 
\begin{equation}\label{eqn_BR_Renyi} 
	\lim_{N\to\infty}\frac1N \sum_{\substack{1\leq n \leq N\\ a(n)=k}} h(T^{\Omega(n)}x)= d_k\cdot \int_X h \,d\nu
\end{equation}
for any $k\ge0$, $h\in C(X)$ and $x\in X$, where $d_k$ is the local density of $a(n)$ of level $k$ defined in \eqref{def_d_k}.
\end{corollary}

In \cite{ErdosIvic1987}, Erd\H{o}s and Ivi\'c also considered the distribution of values of $s$-functions at consecutive integers. More precisely, they proved in \cite[Theorems 1 and 2]{ErdosIvic1987} that for any two $s$-functions $a_1, a_2:\N\to\C$, we have
\begin{align}
&\sum_{\substack{n\leq N\\a_1(n)=a_2(n+1)}} 1 = \rho_1 \cdot N + O(N^{\frac34}\log^4N), \label{ErdosIvic1987_thm1}\\ 
& \sum_{n\leq N} a_1(n)a_2(n+1)= \rho_2 \cdot N + O(N^{\frac34+\ve}), \label{ErdosIvic1987_thm2}
\end{align}
where\footnote{\label{footnote}In \cite{ErdosIvic1987}, there is an extra factor $6/\pi^2$ for the expressions of $\rho_1$ and $\rho_2$. On page 56 of \cite{ErdosIvic1987}, they used the identity $\sum\limits_{\substack{d=1 \\ (d,abc)=1}}^\infty \frac{\mu(d)}{d^2}=\frac6{\pi^2}\prod\limits_{p\nmid ab}(1-p^{-2}) \prod\limits_{p \mid c} (1-p^{-2})^{-1}$ for $(c,ab)=1$, which is equivalent to saying that $\prod\limits_{p\nmid abc}(1-p^{-2}) = \prod\limits_{p\nmid ab}(1-p^{-2}) \prod\limits_{p\nmid c}(1-p^{-2})$. This is not true. Indeed, we have $\prod\limits_{p\nmid abc}(1-p^{-2}) = \frac{\pi^2}{6}\prod\limits_{p\nmid ab}(1-p^{-2}) \prod\limits_{p\nmid c}(1-p^{-2})$ by Lemma~\ref{lem_coprime} in Section~\ref{sec_mainthm_BR_EI_pf}.} 
\begin{align}
	\rho_1 &= \sum_{\substack{f_1,f_2 \in\sF\\ (f_1, f_2)=1\\ a_1(f_1)=a_2(f_2)}}\frac{1}{f_1 f_2}\prod_{p\mid f_1 f_2}(1-\frac{1}{p})\prod_{p\nmid f_1 f_2}(1-\frac{2}{p^{2}}), \label{dfn_rho_1}\\
	\rho_2 &=\sum_{\substack{f_1,f_2 \in\sF\\ (f_1, f_2)=1}}\frac{a_1(f_1)a_2(f_2)}{f_1 f_2}\prod_{p\mid f_1 f_2}(1-\frac{1}{p})\prod_{p\nmid f_1 f_2}(1-\frac{2}{p^{2}}). \label{dfn_rho_2}
\end{align}

The proofs of \eqref{ErdosIvic1987_thm1} and \eqref{ErdosIvic1987_thm2} are analogous to each other. Indeed, \eqref{ErdosIvic1987_thm1} and \eqref{ErdosIvic1987_thm2} may be unified by generalizing the definition of $s$-functions. Let $r\ge1$ be an integer. A function $a: \N^r\to\C$ is called a function with squarefull kernel, or simply an $s$-function\footnote{The $s$-functions are defined as integer-valued arithmetic functions before. In the rest of the article, they
are defined to be complex-valued.}, if it satisfies that $a(n_1,\dots, n_r)=a(f(n_1),\dots,f(n_r))$ and $a(n_1,\dots, n_r) \ll (n_1 \cdots n_r)^\ve$, where $f(n_i)$ is the squarefull part of $n_i$ for $1\leq i\leq r$. For $r=1$, this coincides with the previous definition of $s$-functions. For $r=2$, both $1_{a_1(n_1)=a_2(n_2)}$ and $a_1(n_1)a_2(n_2)$ are $s$-functions of two variables for any two $s$-functions $a_1$ and $a_2$ of one variable. In the following theorem, we will show an analogue of  Corollary~\ref{thm_BR_KR} for $s$-functions of two variables. 

\begin{theorem}\label{mainthm_BR_EI}
	Let  $a: \N^2\to\C$ be an $s$-function. Let $(X, \nu, T)$ be a uniquely ergodic system.  Then we have
	\begin{equation}\label{mainthm_BR_EI_eqn}
		\lim_{N\to\infty}\frac{1}{N}\sum_{n\leq N} a(n,n+1) h(T^{\Omega(n)}x) = \Bigg(\sum_{\substack{f_1,f_2 \in\sF\\ (f_1, f_2)=1}}\frac{a(f_1,f_2)}{f_1 f_2}\prod_{p\mid f_1 f_2}(1-\frac{1}{p})\prod_{p\nmid f_1 f_2}(1-\frac{2}{p^{2}})\Bigg) \Big(\int_X h \,d\nu\Big)
	\end{equation}
	for any $h\in C(X)$ and $x\in X$.
\end{theorem}

Taking $a(n_1,n_2)= 1_{a_1(n_1)=a_2(n_2)}$ and $a_1(n_1)a_2(n_2)$ in Theorem~\ref{mainthm_BR_EI}, we obtain that the sequences in \eqref{ErdosIvic1987_thm1} and \eqref{ErdosIvic1987_thm2} are asymptotically uncorrelated to the orbits in \eqref{BR2022thmA}.

\begin{corollary} Let $a_1, a_2:\N\to\C$ be  two $s$-functions. Let $(X, \nu, T)$ be a uniquely ergodic system. Then we have
\begin{align}
	&\lim_{N\to\infty}\frac{1}{N}\sum_{\substack{n\leq N\\a_1(n)=a_2(n+1)}}  h(T^{\Omega(n)}x) = \rho_1 \cdot \int_X h \,d\nu, \\
	&\lim_{N\to\infty}\frac{1}{N}\sum_{n\leq N} a_1(n)a_2(n+1) h(T^{\Omega(n)}x) = \rho_2 \cdot \int_X h \,d\nu
\end{align}
for any $h\in C(X)$ and $x\in X$, where $\rho_1$ and $\rho_2$ are defined by \eqref{dfn_rho_1} and \eqref{dfn_rho_2}, respectively.
\end{corollary}

\begin{remark} 

Let $k\ge2$ be an integer. Then every integer $n\geq 1$ can be written uniquely as $n=q_k f_k$, where $q_k=q_k(n)$ is $k$-free, $f_k=f_k(n)$ is $k$-full and $(q_k, f_k)=1$. An arithmetic function $a\colon \N\to\Z_{\ge0}$ is called a \textit{function with $k$-full kernel}, or simply a \textit{$k$-full kernel functions}, if $a(n) = a(f_k(n))$ for all $n\ge1$ and $a(n)\ll n^\ve$ for any $\ve>0$. Modifying the proofs in Sections~\ref{sec_mainthm_disjoint}-\ref{sec_mainthm_BR_EI_pf} slightly, one can readily generalize the main results in this paper from functions with squarefull kernel to functions with $k$-full kernel for any integer $k\ge2$.
\end{remark}

This paper is organized as follows. In Section~\ref{sec_background}, we will introduce some background material on uniquely ergodic topological dynamical systems. In Section~\ref{sec_mainthm_disjoint}, we first establish a relation between the standard averages and the averages on the $s$-functions. Then we give an elementary proof of Theorem~\ref{mainthm_disjoint}. In Section~\ref{sec_mainthm_disjoint_variant}, we prove a variant of Theorem~\ref{mainthm_disjoint} for natural numbers coprime to nonnegative and integer-valued $s$-functions. In Section~\ref{sec_mainthm_BR_EI_pf}, we use some ideas in the work \cite{ErdosIvic1987} of Erd\H{o}s and Ivi\'c to prove Theorem~\ref{mainthm_BR_EI}. The technique in the proof is similar  to Theorem~\ref{mainthm_disjoint}. In the end, in Section~\ref{sec_mainthm_BR_EI}, we give a new class of arithmetic functions such that their shifted convolutions are asymptotically uncorrelated to the orbits along $\Omega(n)$ in a uniquely ergodic system.

\section{Concepts and notations} \label{sec_background}

\subsection{Big \texorpdfstring{$O$}{Lg}-notation}

Let $D$ be a set. For two functions $f,g\colon D\to\C$ defined on $D$, we write $f= O(g)$ or $f\ll g$, if there exits a positive constant $C$ such that $|f(x)|\leq C |g(x)|$ for all $x\in D$. The implied constant $C$ may depend on some parameters, say $\ve$ and $k$, but it does not depend on the domain $D$. In the proofs of this paper,  $\ve$ is any positive number that may vary from one line to the next. For example, $3\ve$ may be replaced by $\ve$ in the next line, and this does not have any influence on the proofs. We write $f= O_\ve(g)$ or $f\ll_\ve g$, if the implied constant depends on $\ve$.

\subsection{Topological dynamical systems}\label{subsec_tds}

Let $X$ be a compact metric space, and let $T\colon X\to X$ be a continuous map. Then the pair $(X,T)$ is called a \textit{topological dynamical system}. A Borel probability measure $\nu$ on $X$ is called \textit{T-invariant} if $\nu(T^{-1}A) = \nu(A)$ for all measurable subsets $A\subset X$. By the Bogolyubov-Krylov theorem, every topological dynamical system  $(X, T)$ has at least one $T$-invariant measure. If it admits only one $T$-invariant measure $\nu$, then $(X, \nu, T)$ is called \textit{uniquely ergodic}. It is well-known that $(X, \nu, T)$ is uniquely ergodic if and only if 
\begin{equation}\label{eqn_ergodicity}
	\lim_{N \to \infty} \frac1N\sum_{n=1}^N  h(T^n x) = \int_X h \, d\nu
\end{equation}
holds for  all $x \in X$ and $h\in C(X)$. By Bergelson-Richter's theorem~\eqref{BR2022thmA},   \eqref{eqn_ergodicity} also holds if the orbit $\set{T^nx\colon n\in\N}$ is replaced by the orbit $\set{T^{\Omega(n)}x\colon n\in\N}$ along $\Omega(n)$ for every point $x\in X$.

	In \eqref{BR2022thmA}, take $X=\set{0,1},  T\colon x\mapsto x+1\mod2$, $\mu(\set{0})=\mu(\set{1})=1/2$. This system is called a rotation on two points. It is the simplest uniquely ergodic topological dynamical system. If we take $f\colon \set{0,1}\to\R$, $f(0)=1,f(1)=0$ and the initial point $x=0$, then 
	$T^{\Omega(n)}x=\Omega(n)\mod2$, and
	$$\lim_{N\to\infty}\frac1N\sum_{n=1}^Nf(T^{\Omega(n)}x)=\lim_{N\to\infty}\frac1N\#\set{n\le N\colon \Omega(n)\equiv 0\mod 2}=\frac12.$$
This implies that
\begin{equation*}
	\lim_{N\to\infty}\frac{1}{N}\sum_{n=1}^N\lambda(n)=0,
\end{equation*}
which is equivalent to the PNT, where $\lambda(n)=(-1)^{\Omega(n)}$ is the Liouville function. Thus, the PNT can be recovered from Bergelson-Richter's theorem.

\section{Proof of Theorem~\ref{mainthm_disjoint}} \label{sec_mainthm_disjoint}

In this section, we will use some ideas of Montgomery and Vaughan in \cite{MontgomeryVaughan2007} and the second author in \cite{Wang2025jnt} to give
an elementary proof of Theorem~\ref{mainthm_disjoint}. We first cite a lemma in \cite{Wang2025jnt} on the partial sum of arbitrary bounded arithmetic functions over squarefree numbers.
If $S$ is a finite nonempty set, we define $$\BEu{x\in S}c(x)\colonequals\frac1{|S|}\sum_{x\in S}c(x)$$ for any function $c\colon S\to\C$ on $S$.

\begin{lemma}[{\cite[Lemma~4.1]{Wang2025jnt}}]\label{lem_sqfree_coprime}
Let $c\colon \N\to\C$ be a bounded arithmetic function.	For any $y\ge0$ and $f\ge1$, we have
	\begin{equation}\label{eqn_lem_sqfree_coprime}
		\sum_{\substack{n\leq y\\ (n,f)=1}}\mu^2(n)c(n)= y\sum_{h\in \sH(f)}  \frac{\lambda(h)}{h} \sum_{l\leq \sqrt{\frac{y}{h}}}\frac{\mu(l)}{l^2} \BEu{m\leq \frac{y}{hl^2}} c(hl^2m)+O\Bigg(y^{1/2}\prod_{p\mid f}\big(1-p^{-1/2}\big)^{-1}\Bigg),
	\end{equation}
	where $\sH(f)=\set{h\colon p\mid h\Rightarrow p\mid f}$.
\end{lemma}

The following lemma gives an upper bound on the numbers in $\sH(q)$ for $q\ge1$. 

\begin{lemma}\label{lem_reciprocical_friable}
	For $q\ge 2$ and $H\ge1$, we have
	\begin{equation}\label{eqn_lem_reciprocical_friable}
		\sum_{\substack{h>H\\ h\in \sH(q)}}\frac1h\ll H^{-\frac1{4\log q}}\log q.
	\end{equation}
\end{lemma}
\begin{proof} Let $P^+(h)$ be the largest prime factor of $h$. Then by \cite[Lemma~2.5]{Ford2008}, we have
$$\sum_{\substack{h>H\\ h\in \sH(q)}}\frac1h=\sum_{\substack{h>H\\ p\mid h\Rightarrow p\mid q}}\frac1h\leq \sum_{\substack{h>H\\ P^+(h)\leq q}}\frac1h \ll \exp\set{-\frac{\log H}{4\log q}} \log q,$$
which gives the desired bound.
\end{proof}

The following estimate on  squarefull numbers will be frequently used in the proofs. 

\begin{lemma}\label{squarefull_estimate}
	Suppose $\beta>1/2$. Then for $H\ge1$, we have
	\begin{equation}\label{squarefull_estimate_eqn}
		\sum_{\substack{f>H\\ f\in\sF}}\frac1{f^\beta} \ll H^{\frac{1}{2}-\beta} .
	\end{equation}
\end{lemma}

\begin{proof} Since every squarefull integer $f$ is uniquely of the form $l^2m^3$ with $m$ squarefree, we have
\begin{equation}
	\sum_{\substack{f>H\\ f\in\sF}}\frac1{f^\beta} \leq \sum_{l^2m^3>H} \frac{1}{l^{2\beta}m^{3\beta}}=\sum_{m=1}^\infty\frac{1}{m^{3\beta}} \sum_{l > \sqrt{\frac{H}{m^3}}} \frac{1}{l^{2\beta}} \ll \sum_{m=1}^\infty\frac{1}{m^{3\beta}} \cdot \Big(\sqrt{\frac{m^3}{H}}\Big)^{2\beta-1} =H^{\frac{1}{2}-\beta} \sum_{m=1}^\infty\frac{1}{m^{3/2}},
\end{equation}
which is bounded by $O(H^{\frac{1}{2}-\beta})$, as desired.
\end{proof}

Now, we show an asymptotic formula for the partial sum of  $s$-functions twisted by bounded arithmetic functions. 

\begin{proposition}\label{prop_bc_invariant}
Let $b\colon \N\to\C$ be an $s$-function. Let $c\colon  \N\to\C$ be a bounded arithmetic function. For any $N\ge1$, $1\leq F\leq N$, $H\ge1$ and $1\leq L\leq \sqrt{\frac{N}{FH}}$, we have
\begin{multline}\label{eqn_prop_BRD}
		\frac1N\sum_{1\leq n\leq N} b(n)c(n)=\sum_{\substack{f\leq F\\ f\in \sF}}  \frac{b(f)}{f} \sum_{\substack{h\leq H \\ h\in \sH(f)}}  \frac{\lambda(h)}{h} \sum_{l\leq L}\frac{\mu(l)}{l^2} \BEu{m\leq \frac{N}{fhl^2}}c(fhl^2m)\\
	+O\of{L^{-1}}+O\of{H^{-\frac1{4\log F}}}+O\of{ F^{-1/2+\ve}}+O\of{N^{-1/2+\ve}}.
	\end{multline}
The implied constants in the $O$-terms of \eqref{eqn_prop_BRD} depend only on the supnorm of $|c(n)|$.
\end{proposition}

\begin{proof}
	Since every number $n$ can be uniquely written as $n=fq$ with $(f,q)=1$, $f\in\sF$ and $q\in \sS$, we break the summation into two components:
\begin{equation}\label{eqn_two_components}
	\sum_{n\leq N} b(n)c(n)=\sum_{\substack{f\leq N\\ f\in \sF}}  b(f) \sum_{\substack{q\leq N/f \\q\in \sS \\ (q,f)=1}}c(fq).
\end{equation}

For the inner summation, notice that $q\in \sS$ if and only if $\mu^2(q)=1$. By Lemma~\ref{lem_sqfree_coprime}, we have
\begin{equation}\label{eqn_mainthm_aym_second_term}
	\sum_{\substack{q\leq N/f \\q\in \sS \\ (q,f)=1}}c(fq)= \frac{N}{f}\sum_{h\in \sH(f)}  \frac{\lambda(h)}{h} \sum_{l\leq \sqrt{\frac{N}{fh}}}\frac{\mu(l)}{l^2} \BEu{m\leq \frac{N}{fhl^2}}c(fhl^2m)+O\Big(N^{1/2}f^{-1/2}\prod_{p\mid f}\big(1-p^{-1/2}\big)^{-1}\Big).
\end{equation}

Plug \eqref{eqn_mainthm_aym_second_term} into \eqref{eqn_two_components}. For the error term, by $b(f)\ll f^{\ve}\leq N^{\ve}$ for  any $\ve>0$ and $\big(1-p^{-1/2}\big)^{-1}\leq \big(1-2^{-1/2}\big)^{-1}<4$ for any $p\ge2$, we have
\begin{equation}\label{est b f neg half sum}
	\sum_{\substack{f\leq N\\ f\in \sF}}b(f)N^{1/2}f^{-1/2} \prod_{p\mid f}\big(1-p^{-1/2}\big)^{-1} \ll N^{1/2+\ve}\sum_{\substack{f\leq N\\ f\in \sF}}4^{\omega(f)}f^{-1/2} \ll N^{1/2+\ve}.
\end{equation}
This implies that
\begin{equation}\label{eqn_main_error}
		\frac1N\sum_{1\leq n\leq N} b(n)c(n)=\sum_{\substack{f\leq N\\ f\in \sF}}  \frac{b(f)}{f}\sum_{h\in \sH(f)}  \frac{\lambda(h)}{h} \sum_{l\leq \sqrt{\frac{N}{fh}}}\frac{\mu(l)}{l^2} \BEu{m\leq \frac{N}{fhl^2}}c(fhl^2m)+O\of{N^{-1/2+\ve}}.
\end{equation}

Now, we analyze the main term of \eqref{eqn_main_error}. We cut it off into three parts step by step as follows. Firstly, for any $1\leq F\leq N$, we have
    \begin{align}
	&\quad\Big|\sum_{\substack{F<f\leq N\\ f\in \sF}}  \frac{b(f)}{f}\sum_{h\in \sH(f)}  \frac{\lambda(h)}{h} \sum_{l\leq \sqrt{\frac{N}{fh}}}\frac{\mu(l)}{l^2} \BEu{m\leq \frac{N}{fhl^2}}c(fhl^2m)\Big| \nonumber\\
	&\leq \sum_{\substack{F<f\leq N\\ f\in \sF}}  \frac{|b(f)|}{f} \sum_{h\in \sH(f)}  \frac{1}{h} \sum_{l\leq \sqrt{\frac{N}{fh}}}\frac{|\mu(l)|}{l^2} \Big|\BEu{m\leq \frac{N}{fhl^2}}c(fhl^2m)\Big|  \nonumber \\
	&\ll \sum_{\substack{f>F\\ f\in \sF}} \frac{f^{\ve}}f \sum_{h\in \sH(f)}  \frac{1}{h} = \sum_{\substack{f>F\\ f\in \sF}} \frac1{f^{1-\ve}} \prod_{p\mid f} \big(1+\frac1p+\frac1{p^2}+\cdots\big)  \nonumber\\
	&=\sum_{\substack{f>F\\ f\in \sF}} \frac1{f^{1-\ve}} \prod_{p\mid f} \big(1+\frac1p+O(\frac1{p^2})\big) \ll \sum_{\substack{f>F\\ f\in \sF}} \frac1{f^{1-\ve}} \prod_{p\mid f} \big(1+\frac1p\big).        \label{eqn_mainthm_asy_pf_err2}
    \end{align}

By the fact that 
\begin{equation}\label{a fact on prod p mid f}
    \prod_{p\mid f} \big(1+\frac1p\big) \ll \log\log f \ll f^{\ve}
\end{equation}
and Lemma~\ref{squarefull_estimate}, we  get that
\begin{equation}\label{eqn_mainthm_asy_pf_err1}
		\sum_{\substack{f>F\\ f\in \sF}} \frac1{f^{1-\ve}}  \prod_{p\mid f} \big(1+\frac1p\big) \ll 	\sum_{\substack{f>F\\ f\in \sF}} \frac{f^\ve}{f^{1-\ve}}\ll F^{-1/2+2\ve}.
\end{equation}
	This implies that
\begin{multline}
	\frac1N\sum_{1\leq n\leq N} b(n)c(n)=\sum_{\substack{f\leq F\\ f\in \sF}}  \frac{b(f)}{f}\sum_{h\in \sH(f)}  \frac{\lambda(h)}{h} \sum_{l\leq \sqrt{\frac{N}{fh}}}\frac{\mu(l)}{l^2} \BEu{m\leq \frac{N}{fhl^2}}c(fhl^2m)\\
	+O\of{ F^{-1/2+\ve}}+O\of{N^{-1/2+\ve}}.
	\end{multline}

Secondly, for any $H\ge1$, by Lemma~\ref{lem_reciprocical_friable} we have
\begin{align}
		&\quad \Big|\sum_{\substack{f\leq F\\ f\in \sF}}  \frac{b(f)}{f}\sum_{\substack{h>H \\ h\in \sH(f)}}  \frac{\lambda(h)}{h} \sum_{l\leq \sqrt{\frac{N}{fh}}}\frac{\mu(l)}{l^2} \BEu{m\leq \frac{N}{fhl^2}}c(fhl^2m)\Big| \nonumber\\
		&\ll \sum_{\substack{f\leq F\\ f\in \sF}}  \frac{|b(f)|}{f} \sum_{\substack{h>H \\ h\in \sH(f)}}  \frac{1}{h}\ll \sum_{\substack{f\leq F\\ f\in \sF}}  \frac{|b(f)|}{f} H^{-\frac1{4\log f}}\log f \nonumber\\
		&\leq H^{-\frac1{4\log F}}\sum_{\substack{f\leq F\\ f\in \sF}}   \frac{|b(f)|\log f}{f} \ll H^{-\frac1{4\log F}}.
\end{align}
The last estimation is due to that $\sum_{\substack{f\leq F\\ f\in \sF}}   \frac{b(f)\log f}{f}$ is convergent. Then it follows that
\begin{multline}
	\frac1N\sum_{1\leq n\leq N} b(n)c(n)=\sum_{\substack{f\leq F\\ f\in \sF}}  \frac{b(f)}{f} \sum_{\substack{h\leq H \\ h\in \sH(f)}}  \frac{\lambda(h)}{h} \sum_{l\leq \sqrt{\frac{N}{fh}}}\frac{\mu(l)}{l^2} \BEu{m\leq \frac{N}{fhl^2}}c(fhl^2m)\\
	+O\of{H^{-\frac1{4\log F}}}+O\of{F^{-1/2+\ve}}+O\of{N^{-1/2+\ve}}.
\end{multline}

Thirdly, for any $1\leq L\leq \sqrt{\frac{N}{FH}}$, by the fact that
\begin{equation}
	\sum_{l>L}\frac{|\mu(l)|}{l^2} \ll \frac1{L},
\end{equation}
we have 
\begin{multline}\label{pf_mainthm_BR_RD_error_terms_copy}
	\frac1N\sum_{1\leq n\leq N} b(n)c(n)=\sum_{\substack{f\leq F\\ f\in \sF}}  \frac{b(f)}{f} \sum_{\substack{h\leq H \\ h\in \sH(f)}}  \frac{\lambda(h)}{h} \sum_{l\leq L}\frac{\mu(l)}{l^2} \BEu{m\leq \frac{N}{fhl^2}}c(fhl^2m)\\
	+O\of{L^{-1}}+O\of{H^{-\frac1{4\log F}}}	+O\of{ F^{-1/2+\ve}}+O\of{N^{-1/2+\ve}}.
\end{multline}
This completes the proof of Proposition~\ref{prop_bc_invariant}.
\end{proof}

\begin{theorem}\label{mainthm_sfcn}
Let $b\colon \N\to\C$ be an $s$-function. Then we have
	\begin{equation}
		\sum_{1\leq n\leq N} b(n) = \alpha \cdot N+ O_\ve(N^{1/2+\ve})
	\end{equation}
for any $\ve>0$, where $\alpha$ is defined by \eqref{dfn_alpha}, i.e.,
\begin{equation*}
\alpha=\frac6{\pi^2}\sum_{\substack{f\in \sF }} \frac{b(f)}{f} \prod_{p\mid f} \big(1+\frac1p\big)^{-1}.
\end{equation*}
\end{theorem}

\begin{proof}
By taking $c(n)=1$ in \eqref{eqn_main_error}, we have
\begin{equation}\label{eqn_main_error_2}
		\sum_{1\leq n\leq N} b(n)=N\sum_{\substack{f\leq N\\ f\in \sF}}  \frac{b(f)}{f}\sum_{h\in \sH(f)}  \frac{\lambda(h)}{h} \sum_{l\leq \sqrt{\frac{N}{fh}}}\frac{\mu(l)}{l^2} +O\of{N^{1/2+\ve}}.
\end{equation}
Since $$\sum_{l\leq \sqrt{\frac{N}{fh}}}\frac{\mu(l)}{l^2} =\sum_{l=1}^\infty\frac{\mu(l)}{l^2} - \sum_{l> \sqrt{\frac{N}{fh}}}\frac{\mu(l)}{l^2}=\frac{6}{\pi^2} +O\Big(\sqrt{\frac{fh}{N}}\Big). $$
Plugging it into \eqref{eqn_main_error_2} gives
\begin{equation}\label{eqn_main_error_3}
		\sum_{1\leq n\leq N} b(n)=N\cdot \frac{6}{\pi^2}\sum_{\substack{f\leq N\\ f\in \sF}}  \frac{b(f)}{f}\sum_{h\in \sH(f)}  \frac{\lambda(h)}{h} +O\Bigg(N^{1/2}\sum_{\substack{f\leq N\\ f\in \sF}}  \frac{|b(f)|}{f^{1/2}}\sum_{h\in \sH(f)}  \frac{1}{h^{1/2}}\Bigg) +O\of{N^{1/2+\ve}}.
\end{equation}

Moreover, by \eqref{est b f neg half sum} we can see that
\begin{equation}
	\sum_{\substack{f\leq N\\ f\in \sF}}  \frac{|b(f)|}{f^{1/2}}\sum_{h\in \sH(f)}  \frac{1}{h^{1/2}}=\sum_{\substack{f\leq N\\ f\in \sF}}  \frac{|b(f)|}{f^{1/2}}\prod_{p\mid f}\big(1+\frac1{p^{1/2}}+\frac1p+\cdots\big)\ll N^\ve.
\end{equation}
It follows from \eqref{eqn_main_error_3} that
\begin{equation}\label{eqn_main_error_4}
		\sum_{1\leq n\leq N} b(n)=N\cdot \frac{6}{\pi^2}\sum_{\substack{f\leq N\\ f\in \sF}}  \frac{b(f)}{f} \prod_{p\mid f} \big(1+\frac1p\big)^{-1} + O\of{N^{1/2+\ve}}.
\end{equation}

Then, by Lemma~\ref{squarefull_estimate} we have
\begin{equation}
	\sum_{\substack{f> N\\ f\in \sF}}  \frac{b(f)}{f} \prod_{p\mid f} \big(1+\frac1p\big)^{-1} \ll \sum_{\substack{f> N\\ f\in \sF}}  \frac{|b(f)|2^{\omega(f)}}{f} \ll \sum_{\substack{f> N\\ f\in \sF}}  \frac{1}{f^{1-\ve}} \ll \frac1{N^{1/2-\ve}}.
\end{equation}
Thus, we get 
\begin{equation}
		\sum_{1\leq n\leq N} b(n)=N\cdot \frac{6}{\pi^2}\sum_{f\in \sF}  \frac{b(f)}{f} \prod_{p\mid f} \big(1+\frac1p\big)^{-1} + O_\ve\of{N^{1/2+\ve}}.
\end{equation}
This completes the proof.
\end{proof}

\begin{proof}[Proof of Theorem~\ref{mainthm_disjoint}]

Fix $L, H, F$ in Proposition~\ref{prop_bc_invariant}. For any $f\leq F, h\leq H, l\leq L$, we have
\begin{equation}
	\lim_{N\to\infty}\BEu{m\leq \frac{N}{fhl^2}} c(fhl^2m)= A.
\end{equation}
Taking $N\to\infty$ in \eqref{eqn_prop_BRD} gives
\begin{multline}\label{mainthm_disjoint_pf}
	\lim_{N\to\infty}\frac1N\sum_{1\leq n \leq N} b(n)c(n)= A\cdot \sum_{\substack{f\leq F\\ f\in \sF }} \frac{b(f)}{f} \sum_{\substack{h\leq H\\h\in \sH(f)}}  \frac{\lambda(h)}{h} \sum_{l\leq L}\frac{\mu(l)}{l^2}  \\ +O(L^{-1})+O\of{H^{-\frac{1}{4\log F}}}+ O(F^{-1/2+\ve}).
\end{multline}
Taking $H, F, L\to\infty$ in \eqref{mainthm_disjoint_pf}, respectively, we get that
\begin{equation}
	\lim_{N\to\infty}\frac1N\sum_{1\leq n \leq N} b(n)c(n)= A\cdot \frac6{\pi^2}\sum_{\substack{f\in \sF }} \frac{b(f)}{f} \prod_{p\mid f} \big(1+\frac1p\big)^{-1},
\end{equation}
which is equal to $\alpha\cdot A$ by Theorem~\ref{mainthm_sfcn}.
\end{proof}

\section{A variant of Theorem~\ref{mainthm_disjoint}} \label{sec_mainthm_disjoint_variant}

In 2022, Ding \cite{Ding2022} studied the counting function for numbers $n$ satisfying the property $(n,\Omega(n)-\omega(n))=1$. Based on some ideas of R\'enyi \cite{Renyi1955} and Montgomery-Vaughan \cite{MontgomeryVaughan2007}, he showed the following asymptotic formula
\begin{equation}\label{Ding2022}
	\sum_{\substack{n\leq N \\ (n,\Omega(n)-\omega(n))=1}} 1 = \rho_3 \cdot N + O(N^{\frac12}\log^{\frac43}N),
\end{equation}
where
\[
 \rho_3 = \frac6{\pi^2} \sum_{d=1}^\infty\frac{\mu(d)}{\sigma(d)} \sum_{\substack{f>1 \\ f\in\sF \\ d\mid \Omega(f)-\omega(f)}} \frac{(d,f)}{\sigma_f} \prod_{p\mid (d,f)} (1+\frac1p).
\]
Here, $\sigma(n)=\sum_{d\mid n}d$ and $\sigma_f=f\prod_{p\mid f}(1+\frac1p)$.

Inspired by his work, we study the numbers such that  $(n, a(n))=1$ for an $s$-function $a(n)$. For example, one may take $a(n)$ to be $\Omega(n)-\omega(n)$ or the number of nonisomorphic abelian groups of order $n$. In the following statement, we establish a variant of Theorem~\ref{mainthm_disjoint} for numbers coprime to $s$-functions.

\begin{theorem}\label{mainthm_dyn}
	Let $a\colon\N\to\Z_{\ge0}, b\colon\N\to\C$ be two $s$-functions, $a(n)\ge1$ for $n\ge2$, and let $c(n)$ be a bounded arithmetic function of invariant average $A$ under multiplications. Then we have
	 	\begin{equation}\label{mainthm_dyn_eqn}
		\lim_{N\to\infty}\frac1N\sum_{\substack{1\leq n\leq N\\ (n,a(n))=1}} b(n) c(n) = \rho \cdot A,
	\end{equation}
	where 
	\begin{equation}\label{dfn_rho}
	\rho = \frac6{\pi^2} \sum_{d=1}^\infty\frac{\mu(d)}{\sigma(d)} \sum_{\substack{ f>1_{a(1)=0}\\f\in\sF\\ d\mid a(f)}} \frac{(d,f)b(f)}{\sigma_f} \prod_{p\mid (d,f)} (1+\frac1p).
\end{equation}
\end{theorem}

In this section, we follow the approach in Section~\ref{sec_mainthm_disjoint} to prove Theorem~\ref{mainthm_dyn}. Similar to Proposition~\ref{prop_bc_invariant}, we first establish a relation between the averages over the numbers $n$ such that $(n,a(n))=1$ and the standard averages, where $a(n)$ is an $s$-function.

\begin{proposition}\label{prop_BRD}
Let $a\colon \N\to\Z_{\ge0}, b\colon \N\to\C$ be two $s$-functions, and $a(n)\ge1$ for $n\ge2$. Let $c\colon \N\to\C$ be a bounded arithmetic function. For any $N\ge1$, $1\leq D,F\leq N$, $H\ge1$ and $1\leq K\leq \sqrt{\frac{N}{DFH}}$, we have
\begin{multline}\label{eqn_prop_BRD_coprime}
	\frac1N\sum_{\substack{1\leq n \leq N\\ n\notin \sS\\(n,a(n))=1}} b(n)c(n)= \sum_{d\leq D}  \mu(d) \sum_{\substack{1<f\leq F\\ f\in \sF \\ d\mid a(f)}} \frac{(d,f)b(f)}{df} \sum_{\substack{h\leq H\\h\in \sH(\frac{df}{(d,f)})}}  \frac{\lambda(h)}{h} \sum_{k\leq K}\frac{\mu(k)}{k^2}  \BEu{m\leq \frac{N(d,f)}{dfhk^2}} c(\frac{dfhk^2}{(d,f)}m) \\  +O(K^{-1})+O\of{H^{-\frac{1}{4\log(DF)}}}+ O(F^{-1/2+\ve}) + O(D^{-10})+  O(N^{-1/2+\ve}).
\end{multline}
The implied constants in the $O$-terms of \eqref{eqn_prop_BRD_coprime} depend only on the supnorm of $|c(n)|$. 
\end{proposition}

\begin{proof}

By the fact that $\sum_{d\mid n} \mu(d)=1_{n=1}$, we have
\begin{equation}
	S_0\colonequals\sum_{\substack{1\leq n\leq N, n\notin\sS\\ (n,a(n))=1}} b(n)c(n)= \sum_{\substack{1\leq n\leq N\\ n\notin\sS}} b(n)c(n)  \sum_{d\mid (n,a(n))} \mu(d) =\sum_{d\leq N}\mu(d) \sum_{\substack{1\leq n\leq N\\ n\notin\sS\\ d\mid n, d\mid a(n)}} b(n)c(n).
\end{equation}

 We write $n$ as $n=fq$ with $(f,q)=1$, $f\in\sF$ and $q\in \sS$, then $n\notin \sS$ if and only if $f>1$. Since $a(n)$ and $b(n)$ are $s$-functions, we have $a(n)=a(f)$ and $b(n)=b(f)$. So
 \begin{equation}\label{pf_eqn_prop_BRD_sum}
S_0=\sum_{d\leq N}\mu(d) \sum_{\substack{1\leq fq\leq N\\1<f\in\sF, q\in \sS\\ d\mid fq, d\mid a(f)\\(q,f)=1}}b(f)c(fq)=\sum_{d\leq N}\mu(d) \sum_{\substack{1< f\leq N\\f\in\sF\\ d\mid a(f)}} b(f) \sum_{\substack{q\leq N/f\\q\in \sS\\ d\mid fq\\(q,f)=1}}c(fq).
\end{equation}
 
For the innermost summation in \eqref{pf_eqn_prop_BRD_sum}, we have
\begin{equation}\label{pf_eqn_prop_BRD_innersum}
		\sum_{\substack{q\leq N/f\\ q\in \sS \\ d\mid fq\\ (f,q)=1}} c(fq)= \sum_{\substack{q\leq N/f\\ q\in \sS \\ \frac{d}{(d,f)}\mid q\\ (f,q)=1}} c(fq)=\sum_{\substack{q\leq N/f\\ q = \frac{d}{(d,f)} l\in \sS \\ (f,q)=1}} c(fq)=\sum_{\substack{l\leq \frac{N(d,f)}{df}\\  l\in \sS \\ (l, \frac{df}{(d,f)})=1}} c\big(\frac{df}{(d,f)}l\big)=\sum_{\substack{l\leq \frac{N(d,f)}{df} \\ (l, \frac{df}{(d,f)})=1}} \mu^2(l)c\big(\frac{df}{(d,f)}l\big).
\end{equation}
Then plugging \eqref{pf_eqn_prop_BRD_innersum} into \eqref{pf_eqn_prop_BRD_sum} gives the following decomposition of $S_0$:
\begin{equation}\label{pf_eqn_prop_BRD_decom}
	S_0=  \sum_{d\leq N}  \mu(d) \sum_{\substack{1<f\leq N\\ f\in \sF \\ d\mid a(f)}} b(f)\sum_{\substack{l\leq \frac{N(d,f)}{df} \\ (l, \frac{df}{(d,f)})=1}} \mu^2(l)c\big(\frac{df}{(d,f)}l\big).
\end{equation}

By Lemma~\ref{lem_sqfree_coprime}, we have
\begin{multline}\label{pf_eqn_prop_BRD_sqfree}
	\sum_{\substack{l\leq \frac{N(d,f)}{df} \\ (l, \frac{df}{(d,f)})=1}} \mu^2(l)c\big(\frac{df}{(d,f)}l\big)=\frac{N(d,f)}{df}\sum_{h\in \sH(\frac{df}{(d,f)})}  \frac{\lambda(h)}{h} \sum_{k\leq \sqrt{\frac{N(d,f)}{dfh}}}\frac{\mu(k)}{k^2} \BEu{m\leq \frac{N(d,f)}{dfhk^2}} c(\frac{dfhk^2}{(d,f)}m)\\ +O\Bigg(\sqrt{\frac{N(d,f)}{df}} \prod_{p\mid \frac{df}{(d,f)}}\big(1-p^{-1/2}\big)^{-1}\Bigg).
\end{multline}

Let $S=\frac{S_0}{N}$. Plugging \eqref{pf_eqn_prop_BRD_sqfree} into \eqref{pf_eqn_prop_BRD_decom} gives
\begin{equation}
	S=\sum_{d\leq N}  \mu(d) \sum_{\substack{1<f\leq N\\ f\in \sF \\ d\mid a(f)}} \frac{(d,f)b(f)}{df} \sum_{h\in \sH(\frac{df}{(d,f)})}  \frac{\lambda(h)}{h} \sum_{k\leq \sqrt{\frac{N(d,f)}{dfh}}}\frac{\mu(k)}{k^2} \BEu{m\leq \frac{N(d,f)}{dfhk^2}} c(\frac{dfhk^2}{(d,f)}m)+ \mathcal{R},
	\end{equation}
where
\begin{equation}
	\mathcal{R}\ll \frac{1}{N^{1/2}}\sum_{d\leq N}  \frac{|\mu(d)|}{d^{1/2}}\sum_{\substack{1<f\leq N\\ f\in \sF \\ d\mid a(f)}} \frac{(d,f)^{1/2}|b(f)|}{f^{1/2}} \prod_{p\mid \frac{df}{(d,f)}}\big(1-p^{-1/2}\big)^{-1}.
\end{equation}

To estimate $\cR$, we observe that if $d\mid a(f)$ for $f\leq N$, then by $a(f)\ll f^\ve$ we have $(d,f)\leq d\ll f^\ve$ and $d\ll N^\ve$ for any $\ve>0$. Since $b$ is also an $s$-function, $b(f)\ll f^\ve$, then
\begin{equation}\label{new_observation}
	(d,f)^{1/2}|b(f)|\prod_{p\mid \frac{df}{(d,f)}}\big(1-p^{-1/2}\big)^{-1} \ll f^{3\ve/2} 4^{\omega(df/(d,f))} \ll f^{3\ve/2} (df)^{\ve} \ll f^{4\ve}.
\end{equation}
It follows from $d\ll N^\ve$ that 
\begin{equation}
	\cR\ll  \frac{1}{N^{1/2}}\sum_{d\ll N^\ve}  \frac{1}{d^{1/2}}\sum_{\substack{1<f\leq N\\ f\in \sF \\ d\mid a(f)}}    \frac{f^{4\ve}}{f^{1/2}} \ll N^{-1/2+\ve}.
\end{equation}
For any $1\leq D\leq N$, similar to \eqref{eqn_mainthm_asy_pf_err2} we have
\begin{align}
	&\quad\Big|\sum_{D<d\leq N}  \mu(d) \sum_{\substack{1<f\leq N\\ f\in \sF \\ d\mid a(f)}} \frac{(d,f)b(f)}{df} \sum_{h\in \sH(\frac{df}{(d,f)})}  \frac{\lambda(h)}{h} \sum_{k\leq \sqrt{\frac{N(d,f)}{dfh}}}\frac{\mu(k)}{k^2} \BEu{m\leq \frac{N(d,f)}{dfhk^2}} c(\frac{dfhk^2}{(d,f)}m)\Big| \nonumber \\
	&\leq \sum_{d>D}|\mu(d)| \sum_{\substack{1<f\leq N\\ f\in \sF \\ d\mid a(f)}} \frac{(d,f)|b(f)|}{df}  \sum_{h\in \sH(\frac{df}{(d,f)})}  \frac{1}{h}  \sum_{k\geq1}\frac{1}{k^2} \Big|\BEu{m\leq \frac{N(d,f)}{dfhk^2}} c(\frac{dfhk^2}{(d,f)}m)\Big| \nonumber \\
	&\ll \sum_{d>D} |\mu(d)| \sum_{\substack{1<f\leq N\\ f\in \sF \\ d\mid a(f)}} \frac{(d,f)|b(f)|}{df}  \prod_{p\mid \frac{df}{(d,f)}} (1+\frac1p).\label{pf_eqn_prop_BRD_error_D}
\end{align}

Similar to the argument of \eqref{new_observation}, by \eqref{a fact on prod p mid f} we have
\begin{equation}\label{new_observation2}
	(d,f)|b(f)|\prod_{p\mid \frac{df}{(d,f)}} (1+\frac1p)\ll f^{4\ve}.
\end{equation}
Take $\ve=1/100$. It follows that \eqref{pf_eqn_prop_BRD_error_D} is bounded by
\begin{align}
	\sum_{d>D}  \frac{1}{d}\sum_{\substack{1<f\leq N\\ f\in \sF \\ d\ll f^{1/100}}}    \frac{f^{4\ve}}{f}  &\ll \sum_{d>D} \frac{1}{d} \sum_{\substack{1<f\leq N\\ f\in \sF}} \frac{f^{4\ve}}{(d^{100})^{1/10}\cdot f^{9/10}} \nonumber \\
&\ll  \sum_{d>D} \frac{1}{d^{11}}\sum_{\substack{1<f\leq N\\ f\in \sF}} \frac{1}{f^{9/10-4\ve}} \nonumber \\
&\ll  \sum_{d>D} \frac{1}{d^{11}}\ll D^{-10}. \label{eqn one over d f d larger than D}
\end{align}
This implies that
\begin{multline}
	S=\sum_{d\leq D}  \mu(d) \sum_{\substack{1<f\leq N\\ f\in \sF \\ d\mid a(f)}} \frac{(d,f)b(f)}{df} \sum_{h\in \sH(\frac{df}{(d,f)})}  \frac{\lambda(h)}{h} \sum_{k\leq \sqrt{\frac{N(d,f)}{dfh}}}\frac{\mu(k)}{k^2} \BEu{m\leq \frac{N(d,f)}{dfhk^2}} c(\frac{dfhk^2}{(d,f)}m) \\ + O(D^{-{10}})+ O(N^{-1/2+\ve}).
\end{multline}
		
For any $1< F \leq N$, we have
\begin{align}
	&\quad\Big|\sum_{d\leq D}  \mu(d) \sum_{\substack{F<f\leq N\\ f\in \sF \\ d\mid a(f)}} \frac{(d,f)b(f)}{df} \sum_{h\in \sH(\frac{df}{(d,f)})}  \frac{\lambda(h)}{h} \sum_{k\leq \sqrt{\frac{N(d,f)}{dfh}}}\frac{\mu(k)}{k^2} \BEu{m\leq \frac{N(d,f)}{dfhk^2}} c(\frac{dfhk^2}{(d,f)}m)\Big| \nonumber \\
	&\ll \sum_{d\leq D}  |\mu(d)| \sum_{\substack{F<f\leq N\\ f\in \sF \\ d\mid a(f)}} \frac{(d,f)|b(f)|}{df} \sum_{h\in \sH(\frac{df}{(d,f)})}  \frac{1}{h} \nonumber \\
	&\ll \sum_{d=1}^\infty\frac{1}{d}  \sum_{\substack{f>F\\ f\in \sF \\ d\mid a(f)}} \frac{(d,f)|b(f)|}{f}\prod_{p\mid \frac{df}{(d,f)}}\of{1+\frac{1}{p}}. \label{pf_eqn_prop_BRD_error_F}
\end{align}

For the inner summation in \eqref{pf_eqn_prop_BRD_error_F}, observing that $f\in \sF$ and $a(f)\ll f^\ve$, by \eqref{new_observation2} and Lemma~\ref{squarefull_estimate}  we have
\begin{equation}
    \sum_{\substack{f>F\\ f\in \sF \\ d\mid a(f)}} \frac{(d,f)|b(f)|}{f}\prod_{p\mid \frac{df}{(d,f)}}\of{1+\frac{1}{p}}\ll \sum_{\substack{f>F\\ f\in \sF \\ d^2\ll f^{\ve}}} \frac{f^{5\ve}}{f^\ve f}\ll d^{-2}F^{-1/2+5\ve}.
\end{equation}
This implies that \eqref{pf_eqn_prop_BRD_error_F} is bounded by $O\of{F^{-1/2+\ve}}$ and gives  
\begin{multline}
	S=\sum_{d\leq D}  \mu(d) \sum_{\substack{1<f\leq F\\ f\in \sF \\ d\mid a(f)}} \frac{(d,f)b(f)}{df} \sum_{h\in \sH(\frac{df}{(d,f)})}  \frac{\lambda(h)}{h} \sum_{k\leq \sqrt{\frac{N(d,f)}{dfh}}}\frac{\mu(k)}{k^2} \BEu{m\leq \frac{N(d,f)}{dfhk^2}} c(\frac{dfhk^2}{(d,f)}m) \\ + O\of{F^{-1/2+\ve} } +  O(D^{-10})+ O(N^{-1/2+\ve}).
\end{multline}
For any $H\ge1$, we have
\begin{align}
    &\quad\Big|\sum_{d\leq D}  \mu(d) \sum_{\substack{1<f\leq F\\ f\in \sF \\ d\mid a(f)}} \frac{(d,f)b(f)}{df} \sum_{\substack{h> H\\h\in \sH(\frac{df}{(d,f)})}}   \frac{\lambda(h)}{h} \sum_{k\leq \sqrt{\frac{N(d,f)}{dfh}}}\frac{\mu(k)}{k^2} \BEu{m\leq \frac{N(d,f)}{dfhk^2}} c(\frac{dfhk^2}{(d,f)}m)\Big| \nonumber \\
	&\ll \sum_{d\leq D}  |\mu(d)| \sum_{\substack{1<f\leq F\\ f\in \sF \\ d\mid a(f)}} \frac{(d,f)|b(f)|}{df} \sum_{\substack{h> H\\h\in \sH(\frac{df}{(d,f)})}}  \frac{1}{h} \nonumber \\
	&\ll H^{-\frac{1}{4\log(DF)}} \sum_{d\leq D}  |\mu(d)| \sum_{\substack{1<f\leq F\\ f\in \sF \\ d\mid a(f)}} \frac{(d,f)f^\ve}{df} \log \frac{df}{(d,f)} \nonumber \\
	&\ll  H^{-\frac{1}{4\log(DF)}} .  \label{pf_eqn_prop_BRD_error_H}
   \end{align} 
The second last line in \eqref{pf_eqn_prop_BRD_error_H} follows by Lemma~\ref{lem_reciprocical_friable}. This implies that 
\begin{multline}
	S=\sum_{d\leq D}  \mu(d) \sum_{\substack{1<f\leq F\\ f\in \sF \\ d\mid a(f)}} \frac{(d,f)b(f)}{df} \sum_{\substack{h\leq H\\h\in \sH(\frac{df}{(d,f)})}}  \frac{\lambda(h)}{h} \sum_{k\leq \sqrt{\frac{N(d,f)}{dfh}}}\frac{\mu(k)}{k^2} \BEu{m\leq \frac{N(d,f)}{dfhk^2}} c(\frac{dfhk^2}{(d,f)}m) \\ +O\of{H^{-\frac{1}{4\log(DF)}}}+ O\of{F^{-1/2+\ve} } +  O(D^{-10})+ O(N^{-1/2+\ve}).
\end{multline}

For $1\leq K\leq \sqrt{\frac{N}{DFH}}$,  the following series
\begin{equation}
	\sum_{d\ge1}  |\mu(d)| \sum_{\substack{f>1\\ f\in \sF \\ d\mid a(f)}} \frac{(d,f)|b(f)|}{df} \sum_{h\in \sH(\frac{df}{(d,f)})}  \frac{1}{h}
\end{equation}
is convergent. Hence by $\sum_{k>K}\frac{|\mu(k)|}{k^2}=O(K^{-1})$, we eventually obtain that
\begin{multline}
	S=\sum_{d\leq D}  \mu(d) \sum_{\substack{1<f\leq F\\ f\in \sF \\ d\mid a(f)}} \frac{(d,f)b(f)}{df} \sum_{\substack{h\leq H\\h\in \sH(\frac{df}{(d,f)})}}  \frac{\lambda(h)}{h} \sum_{k\leq K}\frac{\mu(k)}{k^2} \BEu{m\leq \frac{N(d,f)}{dfhk^2}} c(\frac{dfhk^2}{(d,f)}m) \\ +O(K^{-1})+O\of{H^{-\frac{1}{4\log(DF)}}}+ O\of{F^{-1/2+\ve} } +  O(D^{-10})+ O(N^{-1/2+\ve}).\end{multline}
This completes the proof of Proposition~\ref{prop_BRD}. 
\end{proof}

\begin{proof}[Proof of Theorem~\ref{mainthm_dyn}]

First, we split the sum into two parts according to $n\in \sS$ and $n\notin \sS$:
	\begin{equation}\label{eqn_mainthm_asy_total}
		\frac1N\sum_{\substack{1\leq n\leq N\\ (n,a(n))=1}} b(n)c(n)= \frac1N\sum_{\substack{1\leq n\leq N, n\in \sS\\ (n,a(1))=1}} b(1)c(n) +  \frac1N\sum_{\substack{1\leq n\leq N, n\notin\sS\\ (n,a(n))=1}} b(n)c(n) \colonequals S_1(N)+S_2(N).
	\end{equation}
	
For $S_1(N)$, if $a(1)=0$, then $S_1(N)=b(1)c(1)/N$ and so $\lim_{N\to\infty} S_1(N)=0$.  If $a(1)\ge1$, then by \cite[Theorem~1.1]{LWWY2025aa}, we have
\begin{equation}\label{f_1}
	\lim_{N\to\infty} S_1(N)=A\cdot \frac{6}{\pi^2} b(1)\prod_{p\mid a(1)}\big(1+\frac1p\big)^{-1}.
\end{equation}

For $S_2(N)$, fixing $D,F,H,K$ and taking $N \to\infty$ in \eqref{eqn_prop_BRD_coprime}, we get that
\begin{multline}
	\lim_{N\to\infty}S_2(N) =A\cdot\sum_{d\leq D}  \mu(d) \sum_{\substack{1<f\leq F\\ f\in \sF \\ d\mid a(f)}} \frac{(d,f)b(f)}{df} \cdot \\ \cdot \sum_{\substack{h\leq H\\h\in \sH(\frac{df}{(d,f)})}}  \frac{\lambda(h)}{h} \sum_{k\leq K}\frac{\mu(k)}{k^2}  +O(K^{-1})+O\of{H^{-\frac{1}{4\log(DF)}}}+ O\of{F^{-1/2+\ve} } +  O(D^{-10}).
	\end{multline}

Then taking $K, H, F, D\to\infty$, respectively, we get that 
\begin{equation}\label{f_2}
	\lim_{N\to\infty}S_2(N)=A\cdot \frac6{\pi^2}\sum_{d=1}^\infty  \mu(d) \sum_{\substack{f>1\\ f\in \sF \\ d\mid a(f)}} \frac{(d,f)b(f)}{df} \prod_{p\mid \frac{df}{(d,f)}}\big(1+\frac1p\big)^{-1}.
\end{equation}

Combining \eqref{f_1} and \eqref{f_2}, we obtain that
\begin{equation}
	\lim_{N\to\infty} \frac1N\sum_{\substack{1\leq n\leq N\\ (n,a(n))=1}} b(n)c(n)= A\cdot \frac6{\pi^2} \sum_{d=1}^\infty\frac{\mu(d)}{\sigma(d)} \sum_{\substack{ f>1_{a(1)=0}\\f\in\sF\\ d\mid a(f)}} \frac{(d,f)b(f)}{\sigma_f} \prod_{p\mid (d,f)} (1+\frac1p),
\end{equation}
which is equal to $\rho\cdot A$ by the definition of $\rho$ in \eqref{dfn_rho}.
This completes the proof of Theorem~\ref{mainthm_dyn}.
\end{proof}

\begin{remark}
Let $a\colon\N\to\Z_{\ge0}, b\colon\N\to\C$ be two $s$-functions, and $a(n)\ge1$ for $n\ge2$. Similar to the argument of Theorem~\ref{mainthm_sfcn}, one can show that
	\begin{equation}
		\sum_{\substack{1\leq n\leq N\\ (n,a(n))=1}} b(n) = \rho \cdot N+ O_\ve(N^{1/2+\ve})
	\end{equation}
for any $\ve>0$, where $\rho$ is defined in \eqref{dfn_rho}. This generalizes Ding's result \eqref{Ding2022} to $s$-functions with an inferior error term. 
\end{remark}

\section{Proof of Theorem~\ref{mainthm_BR_EI}} \label{sec_mainthm_BR_EI_pf}

In this section, we will prove Theorem~\ref{mainthm_BR_EI}. To prove it, we will use the dynamical generalization of the PNT for arithmetic progressions established by Bergelson and Richter in \cite[Corollary 1.16]{BergelsonRichter2022}. They proved that given a uniquely ergodic topological dynamical system  $(X, \nu, T)$, we have
\begin{equation}\label{BR_pntap} 
	\lim_{N\to\infty}\frac1N \sum_{1\leq n \leq N} h(T^{\Omega(mn+r)}x)= \int_X h \,\nu 
\end{equation}
for any $h \in C(X)$, $x\in X$, $m\in \N$ and $r\in \set{0,1\dots,m-1}$.

For integers \( f_1, f_2 \in\N\) and \( (f_1,f_2)=1 \), we define
$$\cA_{f_1,f_2}(N)=\{q_1\in\N: q_1f_1\leq N, q_2f_2-q_1f_1=1, (q_1,f_1)=(q_2,f_2)=1, \mu^2(q_1)=\mu^2(q_2)=1\}.$$ In \cite{ErdosIvic1987}, to establish \eqref{ErdosIvic1987_thm1} and \eqref{ErdosIvic1987_thm2}, Erd\H{o}s and Ivi\'c estimate the size of $\cA_{f_1,f_2}(N)$.

\begin{lemma}[{\cite[Lemma 1]{ErdosIvic1987}}]\label{lem_ErdosIvic1987}
    Let \( f_1, f_2 \) be two natural numbers such that \( (f_1,f_2)=1 \). Then for \( 1 \leq f_1 \leq \sqrt{N} \), we have
\begin{equation}
	|\cA_{f_1,f_2}(N)|=\rho_4\cdot N+O\Big(N^{\frac{3}{4}}(f_1f_2)^{-\frac{1}{2}}2^{\omega(f_1f_2)}+(\frac{N}{f_1})^{\frac{1}{2}}+(\frac{N}{f_2})^{\frac{1}{2}}\frac{1}{f_1}\Big),
\end{equation}
where\footnote{In \cite{ErdosIvic1987}, there is an extra factor $6/\pi^2$ for the expressions of $\rho_4$; see footnote~\ref{footnote}.}
\begin{equation}\label{dfn_rho_4}
	\rho_4=\frac{1}{f_1f_2}\prod\limits_{p\mid f_1f_2}(1-\frac{1}{p})\prod\limits_{p\nmid f_1f_2}(1-\frac{2}{p^{2}}).
\end{equation}
\end{lemma}

To prove Theorem~\ref{mainthm_BR_EI}, we first show a variant of Bergelson-Richter's theorem over integers in  $\cA_{f_1,f_2}(N)$.

\begin{theorem}\label{BR_A_f_1_f_2}
	Let \( f_1, f_2 \in\N\) and \( (f_1,f_2)=1 \). Let $(X, \nu, T)$ be a uniquely ergodic system. Then  we have 
\begin{equation}\label{BR_A_f_1_f_2_eqn} 
	\lim_{N\to\infty}\frac1N \sum_{q_1\in \cA_{f_1,f_2}(N) } h(T^{\Omega(q_1)}x)= \rho_4\cdot \int_X h \,d\nu
\end{equation}
for any $h\in C(X)$ and $x\in X$, where $\rho_4$ is defined by \eqref{dfn_rho_4}.
\end{theorem}

To prove Theorem~\ref{BR_A_f_1_f_2}, we first prove an elementary result on the partial Euler product of the Riemann zeta function. It will be used to calculate the coefficient $\rho_4$ in \eqref{BR_A_f_1_f_2_eqn}.

\begin{lemma}\label{lem_coprime}

Let $m,n\in \N$ and $(m,n)=1$. Then we have
\begin{equation}\label{lem_coprime_eqn1}
	\prod_{p\nmid mn} (1-\frac1{p^s}) = \zeta(s)\prod_{p\nmid m} (1-\frac1{p^s}) \prod_{p\nmid n} (1-\frac1{p^s})
\end{equation}
and
\begin{equation}\label{lem_coprime_eqn2}
	\prod_{p\nmid mn} (1-\frac1{p^s}) =\prod_{p\nmid m} (1-\frac1{p^s}) \prod_{p\mid n} (1-\frac1{p^s})^{-1}
\end{equation}
for $\Re(s)>1$.
\end{lemma}
\begin{proof}
	By the Euler product of the Riemann zeta function $\zeta(s)$, we have
	\[
	\zeta(s)=\prod_{p}(1-\frac1{p^s})^{-1}.
	\]
Since $(m,n)=1$, we have
\[
\prod_{p\mid mn} (1-\frac1{p^s})=\prod_{p\mid m} (1-\frac1{p^s}) \prod_{p\mid n} (1-\frac1{p^s}).
\]
	Then
	\begin{align*}
		\prod_{p\nmid mn} (1-\frac1{p^s}) &= \frac{\prod_{p} (1-\frac1{p^s})}{\prod_{p\mid mn} (1-\frac1{p^s})} \\
		&=  \frac{\prod_{p} (1-\frac1{p^s})}{\prod_{p\mid m} (1-\frac1{p^s}) \prod_{p\mid n} (1-\frac1{p^s})} \\
		&= \prod_{p}(1-\frac1{p^s})^{-1} \cdot \frac{\prod_{p} (1-\frac1{p^s})}{\prod_{p\mid m} (1-\frac1{p^s})}\cdot  \frac{\prod_{p} (1-\frac1{p^s})}{\prod_{p\mid n} (1-\frac1{p^s})} \\
		&= \zeta(s)\prod_{p\nmid m} (1-\frac1{p^s}) \prod_{p\nmid n} (1-\frac1{p^s}).
	\end{align*}
Moreover, by the second identity above, we have
\begin{equation*}
	\prod_{p\nmid mn} (1-\frac1{p^s}) = \frac{\prod_{p} (1-\frac1{p^s})}{\prod_{p\mid m} (1-\frac1{p^s})}\cdot \frac{1}{\prod_{p\mid n} (1-\frac1{p^s})} = \prod_{p\nmid m} (1-\frac1{p^s}) \prod_{p\mid n} (1-\frac1{p^s})^{-1}.
\end{equation*}
This completes the proof of \eqref{lem_coprime_eqn1} and \eqref{lem_coprime_eqn2}.
\end{proof}

\begin{proof}[Proof of Theorem~\ref{BR_A_f_1_f_2}]

Put $c(n)=h(T^{\Omega(n)}x)$ and $S=\sum_{q_1\in \cA_{f_1,f_2}(N) } c(q_1)$. Then $c(n)$ is bounded, and
\begin{equation*}
	S= \sum_{\substack{q_1,q_2:\,q_1f_1\leq N \\ q_2f_2-q_1f_1=1 \\ (q_1,f_1)=1, (q_2,f_2)=1}} \mu^2(q_1)\mu^2(q_2)c(q_1).
\end{equation*}

First,  using $\mu^2(q_1)=\sum_{m_1^2l_1=q_1}\mu(m_1)$, we have $(m_1,f_1)=(l_1,f_1)=1$. Since $q_2f_2-q_1f_1=1$, we also have $(m_1,f_2)=1$. So
\begin{align}
	S & = \sum_{\substack{m_1\leq \sqrt{\frac{N}{f_1}} \\ (m_1,f_1f_2)=1}} \mu(m_1) \sum_{\substack{l_1\leq \frac{N}{m_1^2f_1}, \,(l_1,f_1)=1 \\ q_2f_2-m_1^2l_1f_1=1 \\  (q_2,f_2)=1}} \mu^2(q_2)c(m_1^2l_1) \nonumber \\
	&= \sum_{\substack{m_1\leq N^{\frac14}f_1^{-\frac12} \\ (m_1,f_1f_2)=1}} ... \quad + \sum_{\substack{ N^{\frac14}f_1^{-\frac12} < m_1 \leq \sqrt{\frac{N}{f_1}}  \\ (m_1,f_1f_2)=1}} ...  \nonumber\\
	&:= S_1 +S_2.
\end{align}

For $S_2$, we have
\begin{align}
	|S_2| &\ll  \sum_{ N^{\frac14}f_1^{-\frac12} < m_1 \leq \sqrt{\frac{N}{f_1}}} \sum_{\substack{l_1\leq \frac{N}{m_1^2f_1} \\ q_2f_2-m_1^2l_1f_1=1 }}1 \nonumber\\
	& =  \sum_{ N^{\frac14}f_1^{-\frac12} < m_1 \leq \sqrt{\frac{N}{f_1}}} \Big(\frac{N}{m_1^2f_1f_2} +O(1)\Big) \nonumber \\
	& \ll_{f_1,f_2} N^{\frac34}.
\end{align}
Here and thereafter, the implied constant may depend on $f_1$ and $f_2$ in this proof.

For $S_1$, we use $\mu^2(q_2)=\sum_{m_2^2l_2=q_2}\mu(m_2)$,  $(m_2,f_2)=(l_2,f_2)=1$ to get that
\begin{equation*}
	S_1 = \sum_{\substack{m_1\leq N^{\frac14}f_1^{-\frac12} \\ (m_1,f_1f_2)=1}} \mu(m_1) \sum_{\substack{m_2\leq \sqrt{\frac{N+1}{f_2}} \\ (m_2,m_1f_1f_2)=1}} \mu(m_2)  \sum_{\substack{l_1\leq \frac{N}{m_1^2f_1} \\ m_2^2l_2f_2-m_1^2l_1f_1=1 \\  (l_1,f_1)=1, (l_2,f_2)=1}}c(m_1^2l_1).
\end{equation*}

To get rid of the coprime restrictions, we use $1_{(l_i,f_i)=1} = \sum_{e_i \mid l_i, e_i\mid f_i} \mu(e_i)$ and write $l_i=d_ie_i$, $i=1,2$. Then
\begin{align}
	S_1& = \sum_{\substack{m_1\leq N^{\frac14}f_1^{-\frac12}\\ (m_1,f_1f_2)=1}} \mu(m_1) \sum_{\substack{m_2\leq \sqrt{\frac{N+1}{f_2}} \\ (m_2,m_1f_1f_2)=1}} \mu(m_2) \sum_{e_1\mid f_1} \mu(e_1) \sum_{e_2\mid f_2} \mu(e_2) \sum_{\substack{d_1\leq \frac{N}{m_1^2e_1f_1} \\ m_2^2d_2e_2f_2-m_1^2d_1e_1f_1=1}}c(d_1e_1m_1^2).
\end{align}

Notice that the number of $d_1\leq \frac{N}{m_1^2e_1f_1}$ such that $m_1^2d_1e_1f_1 \equiv-1 \mod{m_2^2e_2f_2}$ is equal to $\frac{N}{m_1^2m_2^2e_1e_2f_1f_2}+O(1)$. We can write $S_1$ as
\begin{multline}
	S_1 = \frac{N}{f_1f_2}\sum_{e_1\mid f_1} \frac{\mu(e_1)}{e_1} \sum_{e_2\mid f_2} \frac{\mu(e_2)}{e_2} \sum_{\substack{m_1\leq N^{\frac14}f_1^{-\frac12} \\ (m_1,f_1f_2)=1}} \frac{\mu(m_1)}{m_1^2} \sum_{\substack{m_2\leq \sqrt{\frac{N+1}{f_2}} \\ (m_2,m_1f_1f_2)=1}} \frac{\mu(m_2)}{m_2^2} \cdot\\ \cdot\BEu{\substack{d_1\leq \frac{N}{m_1^2e_1f_1} \\ m_2^2d_2e_2f_2-m_1^2d_1e_1f_1=1}}c(d_1e_1m_1^2)  + O(N^{\frac34}).
\end{multline}

Then for any $1\leq H \ll N^{\frac14}$, we have
\begin{multline}\label{BR_A_f_1_f_2_pf}
	S_1 = \frac{N}{f_1f_2}\sum_{e_1\mid f_1} \frac{\mu(e_1)}{e_1} \sum_{e_2\mid f_2} \frac{\mu(e_2)}{e_2} \sum_{\substack{m_1\leq H \\ (m_1,f_1f_2)=1}} \frac{\mu(m_1)}{m_1^2} \sum_{\substack{m_2\leq H \\ (m_2,m_1f_1f_2)=1}} \frac{\mu(m_2)}{m_2^2} \cdot\\ \cdot\BEu{\substack{d_1\leq \frac{N}{m_1^2e_1f_1} \\ m_2^2d_2e_2f_2-m_1^2d_1e_1f_1=1}}c(d_1e_1m_1^2) + O(\frac{N}{H}) + O(N^{\frac34}).
\end{multline}
It follows that
\begin{multline}\label{BR_A_f_1_f_2_pf_final}
	\frac{S}{N} = \frac{1}{f_1f_2}\sum_{e_1\mid f_1} \frac{\mu(e_1)}{e_1} \sum_{e_2\mid f_2} \frac{\mu(e_2)}{e_2} \sum_{\substack{m_1\leq H \\ (m_1,f_1f_2)=1}} \frac{\mu(m_1)}{m_1^2} \sum_{\substack{m_2\leq H \\ (m_2,m_1f_1f_2)=1}} \frac{\mu(m_2)}{m_2^2} \cdot\\ \cdot\BEu{\substack{d_1\leq \frac{N}{m_1^2e_1f_1} \\ m_2^2d_2e_2f_2-m_1^2d_1e_1f_1=1}}c(d_1e_1m_1^2) + O(\frac{1}{H}) + O(N^{-\frac14})
\end{multline}
for any $1\leq H \ll N^{\frac14}$. By \eqref{BR_pntap}, for any fixed $H$ we have
\[
\lim_{N\to\infty}  \BEu{\substack{d_1\leq \frac{N}{m_1^2e_1f_1} \\ m_2^2d_2e_2f_2-m_1^2d_1e_1f_1=1}}c(d_1e_1m_1^2) = \int_X h \,d\nu.
\]

It follows by taking $N\to\infty$ and $H\to\infty$ that
\begin{equation}\label{BR_A_f_1_f_2_pf_rho_4}
	\lim_{N\to\infty} \frac{S}{N}  = \Big(\int_X h \,d\nu\Big) \cdot \frac{1}{f_1f_2}\sum_{e_1\mid f_1} \frac{\mu(e_1)}{e_1} \sum_{e_2\mid f_2} \frac{\mu(e_2)}{e_2} \sum_{\substack{m_1=1 \\ (m_1,f_1f_2)=1}}^\infty \frac{\mu(m_1)}{m_1^2} \sum_{\substack{m_2=1 \\ (m_2,m_1f_1f_2)=1}}^\infty \frac{\mu(m_2)}{m_2^2}.
\end{equation}

Now, we calculate the coefficient. Clearly, we have
\begin{equation}\label{rho_4_pf_eqn1}
	\sum_{e_1\mid f_1} \frac{\mu(e_1)}{e_1} \sum_{e_2\mid f_2} \frac{\mu(e_2)}{e_2} =\prod_{p\mid f_1f_2}(1-\frac1p).
\end{equation}

Since $(m_1,f_1f_2)=1$, by Lemma~\ref{lem_coprime} we have
\begin{equation*}
	\sum_{\substack{m_2=1 \\ (m_2,m_1f_1f_2)=1}}^\infty \frac{\mu(m_2)}{m_2^2} =\prod_{p\nmid m_1f_1f_2} (1-\frac1{p^2}) = \prod_{p\nmid f_1f_2} (1-\frac1{p^2}) \prod_{p\mid m_1} (1-\frac1{p^2})^{-1}.
\end{equation*}
Then
\begin{align}
	&\quad \sum_{\substack{m_1=1 \\ (m_1,f_1f_2)=1}}^\infty \frac{\mu(m_1)}{m_1^2} \sum_{\substack{m_2=1 \\ (m_2,m_1f_1f_2)=1}}^\infty \frac{\mu(m_2)}{m_2^2} \nonumber\\
	&= \Big(\prod_{p\nmid f_1f_2} (1-\frac1{p^2})\Big)\Big(\sum_{\substack{m_1=1 \\ (m_1,f_1f_2)=1}}^\infty \frac{\mu(m_1)}{m_1^2} \prod_{p\mid m_1} (1-\frac1{p^2})^{-1}\Big) \nonumber\\
	&=\prod_{p\nmid f_1f_2} (1-\frac1{p^2})(1-\frac1{p^2} (1-\frac1{p^2})^{-1}) \nonumber\\
	&= \prod_{p\nmid f_1f_2} (1-\frac2{p^2}). \label{rho_4_pf_eqn2}
\end{align}

Combining \eqref{BR_A_f_1_f_2_pf_rho_4}, \eqref{rho_4_pf_eqn1} and \eqref{rho_4_pf_eqn2} together, we conclude that
\[
\lim_{N\to\infty} \frac{S}{N}  = \frac{1}{f_1f_2}\prod\limits_{p\mid f_1f_2}(1-\frac{1}{p})\prod\limits_{p\nmid f_1f_2}(1-\frac{2}{p^{2}}) \cdot\Big(\int_X h \,d\nu\Big).
\]
This completes the proof of Theorem~\ref{BR_A_f_1_f_2}.
\end{proof}

Now, we use Theorem~\ref{BR_A_f_1_f_2} to prove Theorem~\ref{mainthm_BR_EI}.

\begin{proof}[Proof of Theorem~\ref{mainthm_BR_EI}]

Let $f(n)$ be the squarefull part of $n$. Notice that every number $n$ can be written uniquely as $n=qf, (q,f)=1, q\in \sS, f\in \sF$. Put $c(n)=h(T^{\Omega(n)}x)$, then we divide the summation on the left hand side of \eqref{mainthm_BR_EI_eqn} into two parts according to the size of $f(n)$ as follows:
\begin{equation}\label{mainthm_BR_EI_pf_two_parts}
	\sum_{n\leq N} a(n,n+1) c(n)  = \sum_{\substack{n\leq N \\ f(n)\leq N^{\frac12}}} a(n,n+1) c(n) +  \sum_{\substack{n\leq N \\ f(n)> N^{\frac12}}} a(n,n+1) c(n).
\end{equation}
For the second term of \eqref{mainthm_BR_EI_pf_two_parts}, we have
\begin{equation*}
	\sum_{\substack{n\leq N \\ f(n)> N^{\frac12}}} a(n,n+1) c(n) \ll \sum_{\substack{n\leq N \\ f(n)> N^{\frac12}}}n^\ve \ll N^\ve \sum_{\substack{n\leq N \\ f(n)> N^{\frac12}}}1 \ll  N^\ve \sum_{N^{\frac12}<f \leq N} \sum_{q\leq \frac{N}{f}}1 \ll N^{1+\ve} \sum_{f>N^{\frac12}} \frac1f,
\end{equation*}
which is bounded by $O(N^{\frac34+\ve})$ by Lemma~\ref{squarefull_estimate}. Similarly, the following upper bound also holds:
\begin{equation*}
	\sum_{\substack{n\leq N,\,  f(n)\leq N^{\frac12} \\f(n+1)> N^{\frac12}}} a(n,n+1) c(n) =O(N^{\frac34+\ve}).
\end{equation*}
So
\begin{equation}\label{mainthm_BR_EI_pf_err1}
	\sum_{n\leq N} a(n,n+1) c(n)  = \sum_{\substack{n\leq N \\ f(n), f(n+1)\leq N^{\frac12}}} a(n,n+1) c(n) + O(N^{\frac34+\ve}) :=S_1 + O(N^{\frac34+\ve}).
\end{equation}

Let $f_1=f(n), f_2=f(n+1)$ and $n=q_1f_1, n+1=q_2f_2, q_i\in \sS, f_i\in \sF, (q_i,f_i)=1, i=1,2$.
Then $q_2f_2-q_1f_1=1$. Since $a$ is an $s$-function, we have $a(n,n+1)=a(f_1,f_2)$ and $a(f_1,f_2) \ll (f_1f_2)^{\ve/4} \ll N^{\ve/2}$ for $n\leq N$.  For the first term of \eqref{mainthm_BR_EI_pf_err1},\begin{equation*}
	S_1 =  \sum_{\substack{f_1,f_2\in \sF\\f_1,f_2\leq N^{\frac12} \\(f_1,f_2)=1}} a(f_1,f_2)\sum_{\substack{q_1,q_2\in\sS,\, q_1f_1\leq N \\ q_2f_2-q_1f_1=1 \\ (q_1,f_1)=1, (q_2,f_2)=1}}  c(q_1f_1).
\end{equation*}

Then by Lemma~\ref{lem_ErdosIvic1987}, 
\begin{align}
	S_1 &= N\sum_{\substack{f_1,f_2\in \sF\\f_1,f_2\leq N^{\frac12} \\(f_1,f_2)=1}} \frac{a(f_1,f_2)}{f_1f_2}\prod\limits_{p\mid f_1f_2}(1-\frac{1}{p})\prod\limits_{p\nmid f_1f_2}(1-\frac{2}{p^{2}}) \BEu{q_1\in \cA_{f_1,f_2}(N)}c(q_1f_1) \nonumber \\
	& \qquad + O\Bigg(\sum_{\substack{f_1,f_2\in \sF\\f_1,f_2\leq N^{\frac12} }} N^{\frac{\ve}2}\Big(N^{\frac{3}{4}}(f_1f_2)^{-\frac{1}{2}}2^{\omega(f_1f_2)}+(\frac{N}{f_1})^{\frac{1}{2}}+(\frac{N}{f_2})^{\frac{1}{2}}\frac{1}{f_1}\Big)\Bigg) \nonumber \\
	&:=N\cdot S_2+S_3. \label{mainthm_BR_EI_pf_err2}
\end{align}

By the estimate of the error term in \cite[(2.10)]{ErdosIvic1987}, the error term $S_3$ of \eqref{mainthm_BR_EI_pf_err2} is bounded by $O(N^{\frac34+\ve})$. Now we estimate the coefficient $S_2$ of the main term of \eqref{mainthm_BR_EI_pf_err2}. Let $1\leq H\leq N^{\frac12}$. Notice that 
$$\Big| \sum_{\substack{f_1,f_2\in \sF\\H<f_1\leq N^{\frac12},f_2\leq N^{\frac12} \\(f_1,f_2)=1}} \frac{a(f_1,f_2)}{f_1f_2}\prod\limits_{p\mid f_1f_2}(1-\frac{1}{p})\prod\limits_{p\nmid f_1f_2}(1-\frac{2}{p^{2}}) \BEu{q_1\in \cA_{f_1,f_2}(N)}c(q_1f_1) \Big| \ll \sum_{\substack{f_1,f_2\in \sF\\H<f_1\leq N^{\frac12},f_2\leq N^{\frac12}}} \frac{(f_1f_2)^\ve}{f_1f_2},$$
which is bounded by $O(H^{-\frac12+\ve})$ by Lemma~\ref{squarefull_estimate}. It follows that
\begin{equation}\label{mainthm_BR_EI_pf_err3}
	S_2= \sum_{\substack{f_1,f_2\in \sF\\f_1,f_2\leq H \\(f_1,f_2)=1}} \frac{a(f_1,f_2)}{f_1f_2}\prod\limits_{p\mid f_1f_2}(1-\frac{1}{p})\prod\limits_{p\nmid f_1f_2}(1-\frac{2}{p^{2}}) \BEu{q_1\in \cA_{f_1,f_2}(N)}c(q_1f_1) + O(H^{-\frac12+\ve}).
\end{equation}

Combining \eqref{mainthm_BR_EI_pf_err1}-\eqref{mainthm_BR_EI_pf_err3} together, we obtain that
\begin{multline}\label{mainthm_BR_EI_pf_final}
	\frac1N\sum_{n\leq N} a(n,n+1) c(n)  = \sum_{\substack{f_1,f_2\in \sF\\f_1,f_2\leq H \\(f_1,f_2)=1}} \frac{a(f_1,f_2)}{f_1f_2}\prod\limits_{p\mid f_1f_2}(1-\frac{1}{p})\prod\limits_{p\nmid f_1f_2}(1-\frac{2}{p^{2}}) \BEu{q_1\in \cA_{f_1,f_2}(N)}c(q_1f_1)\\ + O(H^{-\frac12+\ve}) + O(N^{-\frac14+\ve}).
\end{multline}

For any fixed $H$ and $f_1\leq H$, by Theorem \ref{BR_A_f_1_f_2} we have
\begin{equation}
	\lim_{N\to\infty} \BEu{q_1\in \cA_{f_1,f_2}(N)}c(q_1f_1) =\int_X h \,d\nu.
\end{equation}
Thus, \eqref{mainthm_BR_EI_eqn} follows by taking $N\to\infty$ and $H\to\infty$ in \eqref{mainthm_BR_EI_pf_final}, respectively.
\end{proof}

\section{A variant of Theorem~\ref{mainthm_BR_EI}} \label{sec_mainthm_BR_EI}

In this section, we consider a new class of functions that are asymptotically uncorrelated to the orbits in \eqref{BR2022thmA}. Let $a,g: \N\to\C$ be two arithmetic functions satisfying $a(n)=\sum_{d\mid n} g(d)$. In general, such form of $a(n)$ is not an $s$-function. In \cite{DengWang2026}, the authors proved that if the series $\sum_{d=1}^\infty \frac{|g(d)|}{d}$ converges,  then $a(n)$ is asymptotically uncorrelated to any bounded function of invariant average. In the following theorem, we establish a variant of Theorem~\ref{mainthm_BR_EI} for such kind of arithmetic functions.

\begin{theorem}\label{mainthm_shifted_conv_sum_BR}
Let $a_1,g_1,a_2,g_2:\N\to\C$ be  arithmetic functions satisfying $g_i(d)\ll d^{-\frac12-\ve}$ for any $\ve>0$ and  $a_i(n)=\sum_{d\mid n} g_i(d)$ for  $i=1,2$.  	Let $(X,\nu, T)$ be uniquely ergodic. Then
	\begin{equation}\label{mainthm_shifted_conv_sum_BR_eqn}
		\lim_{N\to\infty}\frac{1}{N}\sum_{n\leq N} a_1(n)a_2(n+1) h(T^{\Omega(n)}x) = \Big(\sum_{\substack{m,l=1 \\(m,l)=1}}^\infty \frac{g_1(m)g_2(l)}{ml}\Big)\Big(\int_X h \,d\nu\Big)
			\end{equation}
for any $h \in C(X)$ and $x\in X$.
	
\end{theorem}

For example, using Theorem~\ref{mainthm_shifted_conv_sum_BR} for the arithmetic functions $\frac{\varphi(n)}{n}$ and $\frac{n}{\varphi(n)}$, by $$\frac{\varphi(n)}{n} = \sum_{d\mid n}\frac{\mu(d)}{d}\, \text{ and } \,\frac{n}{\varphi(n)} =\sum_{d\mid n} \frac{\mu^2(d)}{\varphi(d)},$$
we have
\begin{align}
	         \lim\limits_{N \to \infty}\frac{1}{N}\sum\limits_{n=1} ^{N}\frac{\varphi(n)} {n}\cdot \frac{\varphi(n+1)}{n+1}h(T^{\Omega{(n)}}x) & =\prod\limits_{p}(1-\frac{2}{p^{2}})\int_X h \,d\nu, \\
	          \lim\limits_{N \to \infty}\frac{1}{N} \sum\limits_{\substack{n\leq N}}\frac{n}{\varphi(n)}\cdot\frac{n+1}{\varphi(n+1)} h(T^{\Omega{(n)}}x) 
          &= \prod_{p}\left(1+\frac{2}{p(p-1)}\right)\int_Xh \,d\nu.
\end{align}

Now we prove Theorem~\ref{mainthm_shifted_conv_sum_BR}.

\begin{proof}[Proof of Theorem~\ref{mainthm_shifted_conv_sum_BR}]
	Let $c(n)=h(T^{\Omega(n)}x)$, $\alpha=\int_X h \,d\nu$, and
	$$S=\sum_{n\leq N} a_1(n)a_2(n+1)c(n).$$
By $a_i(n)=\sum_{d\mid n} g_i(d)$, $i=1,2$, we have
	\begin{equation*}
		S = \sum_{n\leq N} \Big(\sum_{m\mid n}g_1(m)	\Big)\Big(\sum_{l\mid n+1}g_2(l)\Big) c(n).
	\end{equation*}
	Let $n=ms, n+1=lr$, then $lr-ms=1, (m,l)=1$, and $ml\leq N(N+1)$. We can rewrite $S$ as the following multiple summation:
\begin{equation*}
	S=\sum_{\substack{ml\leq N(N+1) \\ (m,l)=1}} g_1(m)g_2(l) \sum_{\substack{r,s:\, ms\leq N \\lr-ms=1}} c(ms)
\end{equation*}
Let $1\leq H \leq N(N+1)$. Since 
\begin{equation*}
	\sum_{\substack{r,s:\,ms\leq N \\lr-ms=1}}1 = \sum_{\substack{s\leq N/m \\ms\equiv -1 \mod{l}}}1 =\frac{N}{ml} + O(1),
\end{equation*}
we get that
\begin{align}
	S & =\sum_{\substack{ml\leq N(N+1) \\ (m,l)=1}} g_1(m)g_2(l) \Big(\frac{N}{ml} + O(1)\Big) \BEu{\substack{s\leq N/m \\ms\equiv -1 \mod{l}}} c(ms) \nonumber\\
	& = N \sum_{\substack{ml\leq N(N+1) \\ (m,l)=1}} \frac{g_1(m)g_2(l)}{ml} \BEu{\substack{s\leq N/m \\ms\equiv -1 \mod{l}}} c(ms) + O\Big(\sum_{ml\leq N(N+1)} |g_1(m)g_2(l)|\Big) \nonumber\\
	& = N \sum_{\substack{ml\leq H \\(m,l)=1}} \frac{g_1(m)g_2(l)}{ml} \BEu{\substack{s\leq N/m \\ms\equiv -1 \mod{l}}} c(ms) \nonumber \\
	&\qquad + O\Bigg(N \sum_{ml > H} \frac{|g_1(m)g_2(l)|}{ml} \Bigg) + O\Big(\sum_{ml\leq N(N+1)} |g_1(m)g_2(l)|\Big). 
\end{align}

Now, we use the assumption that $g_i(d)\ll d^{-1/2-\ve}$ to estimate the last two error terms. For the first error term, we have
\begin{equation}\label{thm_shifted_conv_sum_pf_err1}
	\sum_{ml > H} \frac{|g_1(m)g_2(l)|}{ml} \ll 	\sum_{ml > H} \frac{1}{(ml)^{\frac32+2\ve}} = \sum_{n > H } \frac{\tau(n)}{n^{\frac{3}{2}+2\ve}} \ll \sum_{n > H } \frac{1}{n^{\frac32+\ve}}  \ll \frac1{H^{\frac12+\ve}},
\end{equation}
where $\tau(n)=\sum_{m\mid n}1$ is the divisor function and satisfies $\tau(n)\ll n^\ve$ for all $n\ge1$. 

For the second error term, we have
\begin{equation}\label{thm_shifted_conv_sum_pf_err2}
	\sum_{ml\leq N(N+1)} |g_1(m)g_2(l)| \ll 	\sum_{ml\leq N(N+1)}  \frac{1}{(ml)^{1/2+\ve}} =	\sum_{n\leq N(N+1)}  \frac{\tau(n)}{n^{\frac12+\ve}} \ll \sum_{n\leq N(N+1)}   \frac{1}{n^{\frac12+\frac{\ve}2}} \ll N^{1-\ve}.
\end{equation}

This gives that
\begin{equation}\label{mainthm_shifted_conv_sum_BR_goal}
	\frac{S}N=  \sum_{\substack{ml\leq H \\(m,l)=1}} \frac{g_1(m)g_2(l)}{ml} \BEu{\substack{s\leq N/m \\ms\equiv -1 \mod{l}}} c(ms) +  O\Big(\frac1{H^{\frac12+\ve}}\Big) + O\Big(\frac1{N^\ve}\Big).
\end{equation}

Since 
$$\lim_{N\to\infty} \BEu{\substack{s\leq N/m \\ms\equiv -1 \mod{l}}} c(ms) = \alpha,$$
taking $N\to\infty$ in \eqref{mainthm_shifted_conv_sum_BR_goal}, we obtain that
\begin{equation}\label{mainthm_shifted_conv_sum_BR_goal2}
	\lim_{N\to\infty} \frac{S}N= \alpha \sum_{\substack{ml\leq H \\(m,l)=1}} \frac{g_1(m)g_2(l)}{ml} +  O\Big(\frac1{H^{\frac12+\ve}}\Big).
\end{equation}

Notice that the double series in the first term of \eqref{mainthm_shifted_conv_sum_BR_goal2} is absolutely convergent and bounded by 
$$\sum_{\substack{m,l=1 \\(m,l)=1}}^\infty \frac{|g_1(m)g_2(l)|}{ml} \leq \Big(\sum_{n=1}^\infty\frac{|g(n)|}{n}\Big)^2 \ll \Big(\sum_{n=1}^\infty\frac{1}{n^{\frac{3}{2}+\ve}}\Big)^2.$$
Thus, \eqref{mainthm_shifted_conv_sum_BR_eqn} follows by taking $H\to\infty$ in \eqref{mainthm_shifted_conv_sum_BR_goal2}.
\end{proof}

\begin{remark}\label{thm_shifted_conv_sum}
Let $c(n)=1$ for all $n\ge1$. Take $H=N(N+1)$ in \eqref{mainthm_shifted_conv_sum_BR_goal}, then we have
	\begin{equation}\label{thm_shifted_conv_sum_eqn}
		\sum_{n\leq N} a_1(n)a_2(n+1)=N\sum_{\substack{m,l=1 \\(m,l)=1}}^\infty \frac{g_1(m)g_2(l)}{ml} + O(N^{1-\ve}).
	\end{equation}
Thus, by \eqref{mainthm_shifted_conv_sum_BR_eqn}  $a_1(n)a_2(n+1)$ and the orbit $\{h(T^{\Omega(n)}x)\}_{n\in \N}$ for any $x\in X$ are asymptotically uncorrelated.
\end{remark}

\section*{Acknowledgments}
This work is supported by the National Natural Science Foundation of China (Grant No. 12561001). Shaoyun Yi is supported by the National Natural Science Foundation of China (Nos.~12301016, 12471187).

\end{document}